\documentclass[oneside]{amsart}
\usepackage{etex}
\usepackage[a4paper]{geometry}
\usepackage{float}
\usepackage{amssymb,amsfonts,amsmath,amsthm }
\usepackage{comment}
\usepackage[all,arc]{xy}
\usepackage{enumerate}
\usepackage{mathrsfs,mathtools}
\usepackage{todonotes,booktabs}
\usepackage{stmaryrd}
\usepackage{marvosym}
\usepackage{graphicx}
\usepackage{pdflscape}
\usepackage{booktabs}
\usetikzlibrary{positioning}
  \def\gn#1#2{{$\href{http://groupnames.org/\#?#1}{#2}$}}
\def\gn#1#2{$#2$}  
\tikzset{sgplattice/.style={inner sep=1pt,norm/.style={red!50!blue},char/.style={blue!50!black},
  lin/.style={black!50}},cnj/.style={black!50,yshift=-2.5pt,left=-1pt of #1,scale=0.5,fill=white}}

\usepackage{bbm}
\newcommand\myeq{\mathrel{\overset{\makebox[0pt]{\mbox{\normalfont\tiny\sffamily def}}}{=}}}
\newcommand{\cal}{\mathcal}
\DeclareMathOperator{\Hom}{Hom}
\DeclareMathOperator{\supp}{supp}
\usepackage{tikz}
\usepackage{tikz-cd}
\usetikzlibrary{trees}
\usetikzlibrary[shapes]
\usetikzlibrary[arrows]
\usetikzlibrary{patterns}
\usetikzlibrary{fadings}
\usetikzlibrary{backgrounds}
\usetikzlibrary{decorations.pathreplacing}
\usetikzlibrary{decorations.pathmorphing}
\usetikzlibrary{positioning}

\makeatletter
\newcommand*{\relrelbarsep}{.386ex}
\newcommand*{\relrelbar}{%
  \mathrel{%
    \mathpalette\@relrelbar\relrelbarsep
  }%
}
\newcommand*{\@relrelbar}[2]{%
  \raise#2\hbox to 0pt{$\m@th#1\relbar$\hss}%
  \lower#2\hbox{$\m@th#1\relbar$}%
}
\providecommand*{\rightrightarrowsfill@}{%
  \arrowfill@\relrelbar\relrelbar\rightrightarrows
}
\providecommand*{\leftleftarrowsfill@}{%
  \arrowfill@\leftleftarrows\relrelbar\relrelbar
}
\providecommand*{\xrightrightarrows}[2][]{%
  \ext@arrow 0359\rightrightarrowsfill@{#1}{#2}%
}
\providecommand*{\xleftleftarrows}[2][]{%
  \ext@arrow 3095\leftleftarrowsfill@{#1}{#2}%
}
\makeatother
\usepackage{xcolor} 
\usepackage{graphicx}
\DeclareMathOperator{\rec}{rec}

\usepackage[pagebackref]{hyperref}

\definecolor{dark-red}{rgb}{0.5,0.15,0.15}
\definecolor{dark-blue}{rgb}{0.15,0.15,0.6}
\definecolor{dark-green}{rgb}{0.15,0.6,0.15}
\hypersetup{
    colorlinks, linkcolor=dark-red,
    citecolor=dark-blue, urlcolor=dark-green
}

\newcommand{\stc}{S^3 \langle 3 \rangle}

\newcommand{\cX}{\cal{X}}

\renewcommand*{\backref}[1]{}
\renewcommand*{\backrefalt}[4]{%
  \ifcase #1 %
No citations.
  \or
(cit. on p. #2).%
  \else
(cit on pp. #2).%
  \fi%
}
\usepackage{subfiles}

\usepackage[nameinlink,capitalise,noabbrev]{cleveref}

\newtheorem{thm}{Theorem}[section]
\newtheorem{cor}[thm]{Corollary}
\newtheorem{prop}[thm]{Proposition}
\newtheorem{lem}[thm]{Lemma}

\newtheorem{thmx}{Theorem}

\theoremstyle{definition}
\newtheorem{defn}[thm]{Definition}
\newtheorem{hyp}[thm]{Hypothesis}

\newtheorem{ex}[thm]{Example}

\theoremstyle{remark}
\newtheorem{rem}[thm]{Remark}

\newtheorem*{thm*}{Theorem}

\makeatletter
\let\c@equation\c@thm
\makeatother
\numberwithin{equation}{section}

\DeclareMathOperator{\End}{End}
\DeclareMathOperator{\Vdim}{vdim}
\DeclareMathOperator{\Sol}{Sol}

\DeclareMathOperator{\cA}{\mathcal{A}}
\DeclareMathOperator{\cC}{\mathcal{C}}

\DeclareMathOperator{\gr}{gr}
\DeclareMathOperator{\rank}{rank}

\DeclareMathOperator{\cF}{\mathcal{F}}
\DeclareMathOperator{\cG}{\mathcal{G}}

\DeclareMathOperator{\reg}{Reg}

\DeclareMathOperator{\TrD}{\mathbf{TrDeg}}

\DeclareMathOperator{\prim}{prim}
\DeclareMathOperator{\indec}{indec}
\DeclareMathOperator{\Ext}{Ext}

\DeclareMathOperator{\depth}{depth}

\DeclareMathOperator{\Mod}{Mod}

\DeclareMathOperator{\Ann}{Ann}

\DeclareMathOperator{\res}{res}

\DeclareMathOperator{\im}{Im}

\DeclareMathOperator{\Eq}{Eq}
\usepackage{enumitem}

\DeclareMathOperator{\Nil}{Nil}
\newcommand{\bA}{\mathbf{A}}
\newcommand{\bV}{\mathbf{V}}

\newcommand{\cK}{\mathcal{K}}

\newcommand{\xr}{\xrightarrow}

\newcommand{\Z}{\mathbb{Z}}

\Crefname{figure}{Figure}{Figures}
\Crefname{assu}{Assumption}{Assumptions}
\Crefname{prop}{Proposition}{Propositions}
\Crefname{lem}{Lemma}{Lemmas}
\Crefname{thm}{Theorem}{Theorems}
\Crefname{ex}{Example}{Examples}

\renewcommand{\frak}{\mathfrak}

\DeclareMathOperator{\Inj}{Inj}

\newcommand{\recollement}[5]{
\xymatrix{{#1} \ar[r]|-{#2} & #3 \ar[r]|-{#4} \ar@<1ex>[l]^-{{#2}_!} \ar@<-1ex>[l]_-{{#2}^*} & #5, \ar@<1ex>[l]^-{{#4}!} \ar@<-1ex>[l]_-{{#4}^*}
}}

\newcommand{\cU}{\mathcal{U}}

\newcommand{\F}{\mathbb{F}}

\usepackage{makecell}

\DeclareMathOperator{\Map}{Map}
\title{The topological nilpotence degree of a Noetherian unstable algebra}
\author{Drew Heard}
\address{Department of Mathematical Sciences, Norwegian University of Science and Technology, Trondheim}
\email{drew.k.heard@ntnu.no}

    \newtheoremstyle{TheoremNum}
        {\topsep}{\topsep}              
        {\itshape}                      
        {}                              
        {\bfseries}                     
        {.}                             
        {.5em}                             
        {\thmname{#1}\thmnote{ \bfseries #3}}
    \theoremstyle{TheoremNum}

\DeclareMathOperator{\Spin}{Spin}

\DeclareMathOperator{\CEss}{CEss}
\DeclareMathOperator{\alg}{alg}
\DeclareMathOperator{\rad}{rad}
\DeclareMathOperator{\Sq}{Sq}
\DeclareMathOperator{\GL}{GL}
\DeclareMathOperator{\SL}{SL}
\date{\today}


\begin{document}

\begin{abstract}
We investigate the topological nilpotence degree, in the sense of Henn--Lannes--Schwartz, of a connected Noetherian unstable algebra $R$. When $R$ is the mod $p$ cohomology ring of a compact Lie group, Kuhn showed how this invariant is controlled by centralizers of elementary abelian $p$-subgroups. By replacing centralizers of elementary abelian $p$-subgroups with components of Lannes' $T$-functor, and utilizing the techniques of unstable algebras over the Steenrod algebra, we are able to generalize Kuhn's result to a large class of connected Noetherian unstable algebras. We show how this generalizes Kuhn's result to more general classes of groups, such as groups of finite virtual cohomological dimension, profinite groups, and Kac--Moody groups. In fact, our results apply much more generally, for example, we establish results for $p$-local compact groups in the sense of Broto--Levi--Oliver, for connected $H$-spaces with Noetherian mod $p$ cohomology, and for the Borel equivariant cohomology of a compact Lie group acting on a manifold. Along the way we establish several results of independent interest. For example, we formulate and prove a version of Carlson's depth conjecture in the case of a Noetherian unstable algebra of minimal depth.
\end{abstract}
\setcounter{tocdepth}{1}
\maketitle
\tableofcontents
\section{Introduction}
\subsection{Motivation and overview}
When $G$ is a compact Lie group, or even just a finite group, the mod $p$ cohomology ring $H_{G}^* \coloneqq H^*(BG;\F_p)$ can be extremely complicated. Nonetheless, the global structure of the ring is better understood. This has its origin in Quillen's work on equivariant cohomology \cite{Quillen1971spectrum}. Quillen introduced the category $\bA_G$ of elementary abelian $p$-subgroups of $G$, with morphisms those group homomorphisms induced by conjugation in $G$.  He then proved that the restriction maps induced a morphism
\[
\begin{tikzcd}
q_1 \colon H_G^* \ar{r} & \displaystyle \varprojlim_{E \in \bA_G} H_E^*,
\end{tikzcd}
\]
which is an $\mathcal{F}$-isomorphism, that is, each element in the kernel of $q_1$ is nilpotent, and for each element $y$ in the inverse limit, there exists an integer $n$ with $y^{p^n}$ in the image of $q_1$. Using this, Quillen showed that the Krull dimension of $H_G^*$ is the maximal rank of an elementary abelian $p$-subgroup of $G$.

The cohomology $H_G^*$ has an action of the Steenrod algebra $\cal{A}$, and is in fact an unstable $\cal{A}$-module (see \Cref{sec:tfunctor}). Quillen's theorem can be restated internally in the category of unstable modules over the Steenrod algebra. In fact, Henn, Lannes, and Schwartz \cite{HennLannesSchwartz1995Localizations} do much more than this. The category of unstable modules $\cal{U}$ has a filtration (the nilpotent filtration)
\[
\cal{U} \supseteq \Nil_1 \supseteq \Nil_2 \supseteq \cdots
\]
first introduced by Schwartz \cite{Schwartz1988La}. In general, the category $\Nil_n$ is the smallest localizing subcategory of $\cal{U}$ containing all $n$-fold suspensions of unstable modules (we refer the reader to \Cref{sec:nilpotent_filtratio} for more details, and further characterizations of $\Nil_n$).

Using the general theory of localization in abelian categories, for any unstable module $M$ over the Steenrod algebra there is an associated localization functor $\lambda_n \colon M \to L_nM$ which is localization away from $\Nil_n$. Quillen's map is precisely localization away from $\Nil_1$ for $M = H_G^*$. Henn, Lannes, and Schwartz introduced the following invariant, which we call the topological nilpotence degree of $M$.
\begin{defn}
Let $M$ be an unstable module, then the topological nilpotence degree of $M$ is
\[
  d_0(M) = \inf \{ k \in \mathbb{N} | \lambda_{k+1}M \text{ is a monomorphism}\}.
\]
\end{defn}
For example, $d_0(H_G^*) = 0$ when the cohomology is detected by elementary abelian subgroups, for example, in the case of the mod 2 cohomology of symmetric groups. We note that if $R$ is a Noetherian unstable algebra, then Henn, Lannes, and Schwartz prove that $d_0(R)$ is a finite number.

In \cite{HennLannesSchwartz1995Localizations} Henn, Lannes, and Schwartz gave a rough upper bound for $d_0(H_G^*(X))$, the mod $p$ Borel-equivariant cohomology of a compact Lie group $G$ acting on a manifold $X$. More recently, the case where $X$ is a point has been considered by Kuhn, who proved the following result \cite{Kuhn2007Primitives,Kuhn2013Nilpotence}. In this, if $G$ is a compact Lie group with maximal central elementary abelian $p$-subgroup $C(G)$, we let $e(G)$ denote the top degree of a generator (with respect to a minimal generating set) of the finitely generated $H_G^*$-module $H_{C(G)}^*$, i.e., the top degree of $\F_p \otimes_{H_G^*}H^*_{C(G)}$. 
\begin{thm}[Kuhn]\label{thm:kuhn}
  Let $G$ be a compact Lie group, then
  \[
d_0(H_G^*) \le \max_{E < G}\{ e(C_G(E)) - \dim(C_G(E))\}.
  \]
\end{thm}
The theorem is actually a combination of several results. Kuhn first defines the \emph{central essential ideal}, $\CEss(G)$, of a compact Lie group as the kernel of the map
\[
\xymatrix{
H_G^* \ar[r] & \displaystyle \prod_{C(G) \lneq E} H^*_{C_G(E)},
}
\]
Here the product is taken over those elementary abelian $p$-subgroups $E$ of $G$ for which $C(G)$ is strictly contained in $E$, and the map is the map induced by the inclusions $C_G(E) \le G$. He then shows that
\begin{equation}\label{eq:kuhn1}
d_0(H_G^*) = \max\{ d_0(\CEss(C_G(E))) \mid E < G\}
\end{equation}
and
\begin{equation}\label{eq:kuhn2}
d_0(\CEss(G)) \le e(G) - \dim(G).
\end{equation}
for any compact Lie group $G$. Combining these two results gives \Cref{thm:kuhn}.

We make the following remarks about this theorem.
\begin{enumerate}
  \item As noted by Kuhn, it suffices in \Cref{thm:kuhn} to only consider those $E$ which contain $C(G)$.
  \item By \cite[Theorem 2.30]{Kuhn2013Nilpotence} the central essential ideal $\CEss(G)$ is non-zero if and only the cohomology $H_G^*$ has depth equal to the rank $c(G)$ of the maximal central elementary abelian $p$-group $C(G)$.
   \item The appearance of $-\dim(G)$ in the theorem comes from Symonds' theorem \cite{Symonds2010CastelnuovoMumford} that the Castelnuovo--Mumford regularity $\reg(H_G^*)$ (see \Cref{sec:regindec}) is less than or equal to $-\dim(G)$.
\end{enumerate}
Using these three remarks, one could restate Kuhn's theorem in the following way:
\[
d_0(H_G^*) \le \underset{\substack{C(G) \le E < G \\ \depth(H_{C_G(E)}^*) = c(C_G(E))}}\max \{ e(C_G(E)) + \reg(H_{C_G(E)}^*)\}.
\]
We state it in this way, as this is closer to the generalization we prove below.
\subsection{Unstable algebras and the topological nilpotence degree}
In the previous section we saw that the topological nilpotence degree of $H_G^*$ can be bounded by invariants coming from the cohomology of elementary abelian $p$-subgroups of $G$. In order to generalize this to an arbitrary unstable Noetherian algebra $R$ we need to explain what plays the role of the centralizer of $R$. For this, we use Lannes' $T$-functor \cite{lannes_ihes}.

We recall in \Cref{sec:tfunctor} that for any pair $(E,f)$ such that $E$ is an elementary abelian group and $f$ is a finite morphism $R \to H_E^*$ of unstable algebras, we can produce a new unstable algebra $T_E(R;f)$, along with a canonical map $\rho = \rho_{R,(E,f)} \colon R \to T_E(R;f)$. If $R = H_G^*$, and $E < G$ is an elementary abelian $p$-subgroup, then the fundamental computation of Lannes is that $T_E(H_G^*;\res_{G,E}^*) \cong H_{C_G(E)}^*$, where $\res_{G,E}^* \colon H_G^* \to H_E^*$ is the induced map, and $\rho \colon H_G^* \to H_{C_G(E)}^*$ is simply the map induced by the inclusion $C_G(E) \to G$. Inspired, by this Dwyer and Wilkerson \cite{DwyerWilkerson1992cohomology} used the components of the $T$-functor to define centrality in a Noetherian unstable algebra. In particular, we say that $(E,f)$ is central if $\rho_{R,(E,f)} \colon R \to T_E(R;f)$ is an isomorphism.

Pairs $(E,f)$ (not necessarily central) as considered above naturally assemble into a category $\bA_R$, known as Rector's category (see \Cref{sec:unstable_algebras}). This category has the property that every endomorphism is an isomorphism, and as such the set of isomorphism classes of objects forms a poset, where
\[
[(E,f)] \le [(V,g)] \quad \text{ if and only if } \quad \Hom_{\bA_R}((E,f),(V,g)) \ne \emptyset
\]
Using work of Dwyer and Wilkerson, we prove the following result. 
\begin{thmx}(\Cref{thm:max_central})
  Let $R$ be a connected Noetherian unstable algebra, then there exists a unique (up to isomorphism) maximal central element $(C,g) \in \bA_R$ with respect to the above poset structure.
\end{thmx}
If $R = H_G^*$ for a finite $p$-group $G$ with group-theoretic center $C(G)$, then $C = C(G)$, however this does not hold in general for a compact Lie group. Instead, there is a monomorphism $C(G) \to C$, which need not be an isomorphism in general, see \Cref{ex:centers3} for an example due to Mislin. We refer to a choice of representative for the central element as the center of $R$, and write $(E,f) \subseteq (V,g)$ if $[(E,f)] \le [(V,g)]$. 

We now have the following dictionary between the usual group-theoretic notions and their analogs in the theory of unstable algebras. 
\begin{table}[H]
\begin{tabular}{@{}ll@{}}
\toprule
\textbf{Group theory}                           & \textbf{Unstable algebra}                  \\ \midrule
Group cohomology $H_G^*$                        & Noetherian unstable algebra $R$            \\ 
Quillen category $\bA_G$          & Rector's category $\bA_R$                     \\ 
Cohomology of the centralizer $H^*_{C_G(E)}$    & Component of Lannes $T$-functor $T_E(R;f)$ \\ 
\makecell{Maximal central elementary \\abelian $p$-subgroup, $C(G) <G$}  & Center of $R$, $(C,g) \in \bA_R$                               \\ \bottomrule
\end{tabular}
\end{table}

Inspired by Kuhn's work, the following is the main result of this paper, and is a generalization of \Cref{thm:kuhn} to certain Noetherian connected unstable algebras. We note that the technical hypothesis mentioned in the theorem is always satisfied if $p = 2$ or if $R$ is concentrated in even degrees. Here, if $R$ is an unstable algebra with center $(C,g)$ we let $c(R)$ denote the rank of the $C$, and let $e(R)$ denote the top degree of $\F_p \otimes_{R}T_E(R;f)$. 
\begin{thmx}(\Cref{thm:main_unstable_algebra})\label{thm:thma}
  Let $R$ be a connected Noetherian unstable algebra with center $(C,g)$, and suppose that $T_E(R;f)$ satisfies the assumptions of \Cref{hyp:duflot} for all $(C,g) \subseteq (E,f)$, then
  \[
d_0(R) \le  \underset{\substack{(C,g) \subseteq (E,f) \in \bA_R \\ \depth(T_E(R;f)) = c(T_E(R;f))}}\max \{e(T_E(R;f)) + \reg(T_E(R;f)) \}.
  \]
\end{thmx}
\subsection{The central essential ideal of a Noetherian unstable algebra}
The proof of \Cref{thm:thma} is given by proving the analogs of \eqref{eq:kuhn1} and \eqref{eq:kuhn2} for an arbitrary connected Noetherian unstable algebra. To do this, we first define the central essential ideal of a Noetherian unstable algebra $R$ with center $(C,g)$ as the unstable algebra fitting in the left exact sequence
\[
0 \to \CEss(R) \to R \to \prod_{(C,g) \subsetneq (E,f)} T_E(R;f)
\]
where the product is taken over the maps $\rho_{R,(E,f)}$. This does not depend on the choice of representative for the center of $R$.

For $G$ a finite group, Kuhn has proved that the Krull dimension of $\CEss(G)$ is at most the rank of $C$. The proof uses a result about transfers due to Carlson \cite{Carlson1995Depth} that is not available for a general unstable algebra. We instead use $\cal{U}$-technology to prove the following result, which is crucial in the sequel.
\begin{thmx}(\Cref{thm:krulldimension}).\label{thm:thmb}
  Let $R$ be a connected Noetherian unstable algebra with center $(C,g)$, then the Krull dimension of $\CEss(R)$ is at most the rank of $C$.
\end{thmx}
This theorem is used crucially in the next result, with is the analog of \eqref{eq:kuhn2}. If $R$ is a Noetherian unstable algebra with center $(C,g)$, then the image of $g \colon R \to H_C^*$ is either a polynomial algebra (when $p = 2$) or a polynomial tensor an exterior algebra (when $p >2$). In particular, there always exists a subalgebra $B \subset R$ such that $B \to \im(g)$ is an isomorphism. Borrowing terminology from Kuhn, we call such a $B$ a \emph{Duflot algebra}. The technical hypothesis \Cref{hyp:duflot} mentioned previously is that the Duflot algebra is  polynomial which, as noted, is automatic if $p = 2$ of if $R$ is concentrated in even degrees. Our analog of \eqref{eq:kuhn2} is the following. 
\begin{thmx}(\Cref{thm:d0cess} and \Cref{thm:cessnonzero})\label{thm:thmc}
 Let $R$ be a connected Noetherian unstable algebra at the prime $p$ with center $(C,g)$ satisfying \Cref{hyp:duflot}, then if $\CEss(R) \ne 0$ we have
  \[
d_0(\CEss(R)) \le e(R) + \reg(R).
  \]
  Moreover, $\CEss(R) \ne 0$ if and only if $\depth(R) = \rank(C)$. In this case, $\CEss(R)$ is a Cohen--Macaulay $R$-module of dimension $\rank(C)$.
\end{thmx}
The statement that if $\depth(R) = \rank(C)$, then $\CEss(R) \ne 0$ can be considered a form of Carlson's depth conjecture (see \cite[Question 12.5.7]{CarlsonTownsleyValeriElizondoZhang2003Cohomology}) in the case of a Noetherian unstable algebra of minimal depth, see also the discussion in \Cref{sec:regindec}. Indeed, we always have $\depth(R) \ge \rank(C)$ by the author's generalized version of Duflot's theorem \cite{heard_depth}, see also \Cref{cor:duflot} in this paper (Carlson considers the case $R = H_G^*$ for $G$ a finite group).

The proof of \Cref{thm:thma} then follows the same strategy as Kuhn; we show in \Cref{prop:kuhn2.7} that for any connected Noetherian unstable algebra $R$ with center $(C,g)$ we have
  \[
d_0(R) \le \max_{(C,g) \subseteq (E,f) \in \bA_R}\{d_0(\CEss(T_E(R;f))) \}.
  \]
  Combining this with the bound coming from \Cref{thm:thmc} then gives the result. 
\subsection{The topological nilpotence degree for the mod $p$ cohomology of groups}
The components of Lannes $T$-functor have been computed for the mod $p$ cohomology of a large number of classes of groups, not just for compact Lie groups. In all these cases, Rector's category $\bA_{H_G^*}$ can be identified with Quillen's category $\bA_G$ with objects the elementary abelian $p$-subgroups of $G$, and central elements in $\bA_{H_G^*}$ correspond to elementary abelian $p$-subgroups $E < G$ for which $C_G(E) \to E$ is a mod $p$ cohomology isomorphism. Borrowing terminology from Mislin \cite{Mislin1992Cohomologically}, we call such subgroups cohomologically $p$-central. Our results imply that there is (up to isomorphism) a unique maximal cohomologically $p$-central subgroup $C_p(G)$, whose rank may be greater than the rank of the usual group-theoretic center of $G$.

\Cref{thm:thma} then gives rise to the following computation of the topological nilpotence degree of the mod $p$ cohomology of these groups.
\begin{thmx}(\Cref{thm:main_groups})
  Assume we are in one of the following cases:
\begin{enumerate}
  \item $G$ is a compact Lie group.
  \item $G$ is a discrete group for which there exists a mod $p$ acyclic $G$-CW complex with finitely many $G$-cells and finite isotropy groups.
    \item $G$ is a profinite group such that the continuous mod $p$ cohomology $H_G^*$ is finitely generated as an $\F_p$-algebra.
    \item $G$ is a group of finite virtual cohomological dimension such that $H_G^*$ is finite generated as an $\F_p$-algebra.
    \item $G$ is a Kac--Moody group.
  \end{enumerate}
Then, for any prime $p$ we have
  \[
d_0(H_G^*) \le  \underset{\substack{C_p(G) \le E \in \bA_G \\ \depth(H_{C_{{G}}(E)}^*) = c(C_{{G}}(E))}} \max \{e(H_{C_{{G}}(E)}^*) + \reg(H_{C_{{G}}(E)}^*)\}
  \]
  where $c(C_G(E))$ is the rank of the maximal cohomologically $p$-central subgroup of $G$.
\end{thmx}
Of course, by including additional summands, one can rewrite this as
  \[
d_0(H_G^*) \le \max_{E < G}\{ e(H^*_{C_G(E)}) + \reg(H^*_{C_G(E))})\}
  \]
  to give a result analogous to \Cref{thm:kuhn}.

   We have similar results in the case of the mod $p$ cohomology of $p$-local compact groups \cite{BrotoLeviOliver2007Discrete}, see \Cref{sec:homotopical}.
     \begin{ex}
  In \Cref{ex:gl2}, we compute that $1 \le d_0(H^*_{\GL_2(\Z_3)}) \le 2$ when $p = 3$. Similarly, in \Cref{ex:s2} we compute that $d_0(H^*_{S_2}) = 2$ at the prime 3, where $S_2$ is the Morava stabilizer group which features prominently in the chromatic approach to stable homotopy theory. \qed
\end{ex}
Finally, in an appendix, we show that a slight variation of our methods shows the following.
\begin{thmx}(\Cref{thm:borel_appendix})
   Let $G$ be a compact Lie group, $X$ a manifold, and suppose that the Duflot algebra for $H_{C_G(E)}^*(X^E)$ is polynomial for all $C(G;X) \le E$, then
  \[
d_0(H_G^*(X)) \le \max_{C(G,X) \le E < G}\{e(C_G(E),X^E) + \dim(X^E) - \dim(C_G(E)) \}
  \]
\end{thmx}
\subsection*{Notation}
The following is some of the notation used in this paper.
\medskip

\begin{tabular}{l|l}
$\cal{U}$ & The category of unstable modules over the Steenrod algebra (\Cref{sec:unstable_algebras}) \\
$\cal{K}$ & The category of unstable algebras over the Steenrod algebra  (\Cref{sec:unstable_algebras}) \\
$R$ & Generic unstable algebra (\Cref{sec:unstable_algebras}) \\
$E$ & Elementary abelian $p$-group \\
$\bA_R$ & Rector's category associated to a Noetherian unstable algebra $R$ (\Cref{sec:unstable_algebras}) \\
$(E,f)$ & Element of Rector's category $\bA_R$ (\Cref{sec:unstable_algebras}) \\
$T_E$ & Lannes' $T$-functor (\Cref{sec:tfunctor})\\
$d_0M$ & Topological nilpotence degree of an unstable module (\Cref{sec:nilpotent_filtratio})\\
$\CEss(R)$ & The central essential ideal of a Noetherian unstable algebra (\Cref{sec:cess}) \\
$P_CM$ & The module of primitives for a comodule (\Cref{sec:prim_indec})  \\
$Q_BM$ & The space of indecomposables for a $B$-module $M$ (\Cref{sec:prim_indec}) \\
$\reg(M)$ & The regularity of a module $M$ (\Cref{sec:regindec}) \\
$\cal{F}$& Fusion system associated to a discrete $p$-toral group $S$ (\Cref{sec:homotopical}) \\
$\cal{F}^e$ & Full subcategory of $\cal{F}$ consisting of fully centralized  \\
&  elementary abelian $p$-subgroups of $S$ (\Cref{sec:homotopical}) \\
$ H_{\frak m}^i(M)$ & The local cohomology of a module $M$ (\Cref{sec:appendix}) \\
\end{tabular}
\subsection*{Conventions}
We will always write $H_G^*(X)$ for the mod $p$ $G$-equivariant cohomology of a space $X$. In particular, taking $X$ to be a point, then $H_G^*$ denotes the group cohomology of $G$. For a space $X$ we will always write $H^*(X)$ for the mod $p$ cohomology of $X$; thus $H_G^* = H^*(BG)$. If $R$ is an augmented $\F_p$-algebra we will write $\epsilon_R \colon R \to \F_p$ for the canonical map; in the case of $R=H^*(X)$, we will often abbreviate this to $\epsilon_X$, or even $\epsilon_G$ if $X = BG$ .
\subsection*{Acknowledgements}
We are grateful to Nick Kuhn both for enlightening conversations, and for the papers \cite{Kuhn2007Primitives,Kuhn2013Nilpotence} from which this work is directly inspired. We also thank Hans--Werner Henn, Niko Naumann, and Burt Totaro for helpful conversations. The author was supported by the `SFB 1085 Higher Invariants' at Universit\"at Regensburg. We are indebted to the referee, whose suggestions, simplifications, and corrections have led to a much improved version of this paper. Revisions of this paper were done while the author was supported by grant number TMS2020TMT02 from the Trond Mohn Foundation.
\section{Noetherian unstable modules, unstable algebras, and Lannes' \texorpdfstring{$T$}{T}-functor}
We being with a review of the theory of unstable modules, unstable algebras, and Lannes' $T$-functor. We introduce the fundamental category $\bA_R$, also known as Rector's category, of a Noetherian unstable algebra $R$. Finally, we review Schwartz's nilpotent filtration of the category of unstable modules.
\subsection{Unstable modules, unstable algebras and Rector's category}\label{sec:unstable_algebras}
Much of this section is well-known, and a useful reference is \cite{schwartz_book}. We first start with the definition of the categories of unstable modules and unstable algebras over the mod $p$ Steenrod algebra. We let $\cal{A}$ denote the mod $p$ Steenrod algebra, for which we assume the reader is familiar with.
\begin{defn}
  An unstable $\cal{A}$-module $M$ is a graded $\cal{A}$-module such that for all $x \in M$
  \begin{enumerate}
    \item $\Sq^ix = 0$ for $i > |x|$, if $p = 2$;
    \item $\beta^eP^ix = 0$ for all $2i+e > |x|$, if $p$ is odd and $e \in \{ 0, 1\}$.
  \end{enumerate}
  We let $\cal{U} \subset \Mod_{\cal{A}}$ denote the full subcategory of graded $\cal{A}$-modules whose objects are unstable $\cal{A}$-modules.
  \end{defn}
We observe that if $M \in \cal{U}$, then $M$ is trivial in negative degrees. If $M^0 \cong \F_p$, then we say the $M$ is \emph{connected}. The category of unstable modules has a suspension functor $\Sigma \colon \cal{U} \to \cal{U}$: given an $\cal{A}$-module $M$, we define $(\Sigma M)^n \cong M^{n-1}$, with $\cal{A}$-module structure given by $\theta(\Sigma m) = (-1)^{|\theta|}\Sigma \theta (m)$ for all $m \in M, \theta \in \cal{A}$.

The mod $p$ cohomology of a space $H^*(X)$ is always an unstable module. In fact, it also has an algebra structure satisfying certain properties, which leads to the following definition.
\begin{defn}
  An unstable $\cal{A}$-algebra $R$ is an unstable $\cal{A}$-module, together with maps $\mu \colon R \otimes R \to R$ and $\eta \colon \F_p \to R$ which determine a commutative, unital, $\F_p$-algebra structure on $R$ and such that the Cartan formula holds (equivalently, $\phi$ is $\cal{A}$-linear) and
  \begin{equation}
\begin{split}\label{eq:unstable}
  \Sq^nx = x^2 & \text{ if } p = 2 \text{ and } n = |x|, \\
  P^nx = x^p & \text{ if } p > 2 \text{ and } 2n = |x|.
\end{split}
\end{equation}
  We let $\cal{K}$ denote the category of unstable algebras over $\cal{A}$. This is the category with objects unstable algebras, and morphisms degree preserving maps which are both $\cal{A}$-linear and maps of graded algebras.

  Finally, we say that $R$ is a Noetherian unstable algebra if $R$ is finitely generated as an algebra.
\end{defn}
\begin{ex}\label{ex:elemn_abelian}
  The mod-$p$ cohomology of an elementary abelian $p$-group $E$ of rank $n$ is of fundamental importance in the theory of unstable algebras over the Steenrod algebra. We recall that
  \[
H_E^* \cong \F_2[x_1,\ldots,x_n]
  \]
  with $|x_i| = 1$ when $p = 2$, and
  \[
H_E^* \cong \F_p[\beta(y_1),\ldots,\beta(y_n)] \otimes \Lambda_{\F_p}(y_1,\ldots,y_n)
  \]
  where $|y_i| = 1$ and $\beta$ denotes the Bockstein homomorphism associated to the sequence $0 \to \Z/p \to \Z/p^2 \to \Z/p \to 0$. In particular, $H_E^*$ is a Gorenstein ring of dimension $n$. Its importance comes from the fact that it is an injective object in the category $\cal{U}$, see \cite{Carlsson1983G,Miller1984Sullivan,LannesZarati1986Sur}.

Finally, we note that the group homomorphism $E \times E \to E$ given by multiplication induces a homomorphism $H_E^* \to H^*_{E \times E} \cong H_E^* \otimes H_E^*$, making $H_E^*$ into a primitively generated Hopf algebra.
  \qed
\end{ex}
Given an unstable algebra $R$, we can also define a category $R-\cU$, whose objects are unstable $\cA$-modules $M$ together with $\cal{A}$-linear structure maps $R \otimes M \to M$ which make $M$ into an $R$-module, and whose morphisms are the $\cal{A}$-linear maps which are also $R$-linear. The full subcategory of $R-\cU$ consisting of the finitely generated $R$-modules will be denoted $R_{fg}-\cU$.
\begin{ex}\label{ex:borel}
  Let $G$ be a compact Lie group and $X$ a manifold, then the Borel equivariant cohomology $H_G^*(X)$ is an object of $R_{fg}-\cU$ for $R = H_G^*$, see \cite{Quillen1971spectrum}.
\end{ex}
The following categories, first studied by Rector \cite{Rector1984Noetherian}, will play a crucial role in the sequel.
\begin{defn}
  Let $R$ be a Noetherian unstable algebra, then the category $\bV_R$ is the category with objects $(E,f)$ where $E$ is an elementary abelian $p$-group, and $f \colon R \to H_E^*$ is a homomorphism of unstable algebras. A morphism $\alpha \colon (E,f) \to (V,g)$ is a morphism $\alpha^* \colon H_V^* \to H_E^*$ of unstable algebras (equivalently, a group homomorphism $\alpha \colon E \to V$) such that the diagram
  \[
\begin{tikzcd}
                              & R \arrow[ld, "f"'] \arrow[rd, "g"] &       \\
H_E^*  &                                    &\arrow[ll, "\alpha^*"]  H_V^*
\end{tikzcd}
  \]
  commutes.

  Rector's category $\bA_R$ is the full subcategory of $\bV_R$ consisting of those $(E,f)$ where $f \colon R \to H_E^*$ is a \emph{finite} morphism, i.e., $H_E^*$ is a finitely generated $R$-module via $f$.
\end{defn}
We observe that if $\alpha \colon (E,f) \to (V,g)$ is a morphism in $\bA_R$, then $\alpha^* \colon H_E^* \to H_V^*$ necessarily arises form a monomorphism $E \to V$ of elementary abelian $p$-groups. We have the following properties of $\bA_R$, where we recall that a Noetherian unstable algebra always has finite Krull dimension.
\begin{prop}\label{prop:rector_props}
  Let $R$ be a Noetherian unstable algebra of Krull dimension $d$.
  \begin{enumerate}
    \item The category $\bA_R$ has a finite skeleton.
    \item For each $(E,f) \in \bA_R$ we have $\rank(E) \le d$. In fact,
    \[
d = \max\{ \rank(E) \mid (E,f) \in \bA_R \}.
    \]
  \end{enumerate}
\end{prop}
\begin{proof}
  Part (1) is due to Rector \cite[Proposition 2.3(1)]{Rector1984Noetherian}, while (2) is an algebraic consequence of Rector's $\cal{F}$-isomorphism theorem \cite[Theorem 1.4]{Rector1984Noetherian}, as extended to the case $p >2 $ by Broto and Zarati \cite{BrotoZarati1988Nillocalization}.
\end{proof}
\begin{rem}\label{rem:kernel}
  Given a pair $(E,f) \in \bV_R$, choosing an element $e \in E$ is equivalent to giving a homomorphism $\chi_e \colon \Z/p \to E$ with $\chi_e(1) = e$. Let $f_e \colon R \to H_{\Z/p}^*$ denote the composite $R \xr{f} H_E^* \xr{\chi_e^*} H_{\Z/p}^*$. Then, the kernel of $f$, denoted $\ker(f)$, is the set consisting of all $e \in E$ with the property that $f_e$ is trivial above dimension 0 \cite[Definition 4.3]{DwyerWilkerson1992cohomology}. By \cite[Proposition 4.4]{DwyerWilkerson1992cohomology}, the pair $(E,f) \in \bA_R$ (that is, the morphism $f \colon R \to H_E^*$ is finite) if and only if $\ker(f) = \{ 0 \}$.

  Moreover, if $R$ is connected and Noetherian, then for any pair $(E,f) \in \mathbf{V}_R$, $\ker(f)$ is a subgroup of $E$, and $f \colon R \to H_E^*$ extends uniquely to a map $\tilde f \colon R \to H_{E/\ker(f)}^*$ such that the pair $(E/\ker(f),\tilde f)$ is in $\bA_R$ \cite[Proposition 4.8]{DwyerWilkerson1992cohomology}. Here, `extends' means that the evident diagram
\[
\begin{tikzcd}
  &R \arrow[swap]{dl}{f} \arrow{dr}{\tilde f}&\\
  H_E^*  && \arrow{ll} H_{E/\ker(f)}^*
\end{tikzcd}
\]
commutes. This construction is functorial; the assignment $(E,f) \mapsto (E/\ker(f),\tilde f)$ defines a functor $\text{rec} \colon \bV_R \to \bA_R$, see \cite[Section 4.6]{Henn1996Commutative} for further discussion.
\end{rem}
\begin{rem}\label{rem:endvd-sets}
  An extension of the work of Rector to the case of unstable algebras of finite transcendence degree $d$ is given by Henn, Lannes, and Schwartz in \cite[Part II]{HennLannesSchwartz1993categories}. Let $V_d = (\Z/p)^d$, considered as a profinite right $\End V_d$-set i.e., a profinite set with a continuous right action of the monoid $\End V_d$. Let $\cal{PS}-\End V_d$ denote the category whose objects are profinite right $\End V_d$-sets, and whose morphisms are maps of profinite sets respecting the $\End V_d$-action, and let $\cal{K}_d$ denote the category of unstable algebras of transcendence degree $d$. In \cite[Theorem II.2.4]{HennLannesSchwartz1993categories} Henn, Lannes, and Schwartz prove that the functor
  \[
s_d \colon \cal{K}_d \to (\cal{PS}-\End V_d)^{\text{op}}, \quad R \mapsto \Hom_{\cal{K}}(R,H_{V_d}^*)
  \]
  induces an equivalence of categories $\cal{K}_d/\Nil_1 \to (\cal{PS}-\End V_d)^{\text{op}}$, where the inverse equivalence is induced by the functor
  \[
b_d \colon (\cal{PS}-\End V_d)^{\text{op}} \to \cal{K}_d, \quad S \mapsto \Hom_{\cal{PS}-\End V_d}(S,H_{V_d}^*).
  \]
  Here, the category $\cal{K}/\Nil_1$ is the quotient category of $\cal{K}$ given by inverting all the $\cal{F}$-isomorphisms. In particular, the natural map $R \to (b_d \circ s_d)(R)$ is an $\cal{F}$-isomorphism for all unstable algebras $R \in \cal{K}_d$.

  Moreover, if $S$ is a Noetherian $\End V_d$-set in the sense of \cite[Definition 5.8]{HennLannesSchwartz1993categories}, then $b_d(S)$ is a Noetherian unstable algebra, and conversely if $R$ is a Noetherian unstable algebra, then $s_d(R)$ is a Noetherian $\End V_d$-set \cite[Theorem 7.1]{HennLannesSchwartz1993categories}. Moreover, to such an $S$, one can associate a category $\cal{R}(S)$ which, in the case where $S = s_d(R)$ for a Noetherian unstable algebra $R$, is Rector's category $\bA_R$, see the remark on page 1097 of \cite{HennLannesSchwartz1993categories}.  Finally, Henn, Lannes, and Schwartz define the notion of the kernel of an element of an $\End V_d$-set, see \cite[Section 5.2]{HennLannesSchwartz1993categories}. If $R$ is a connected Noetherian unstable algebra, and $(E,f) \in \bA_R$, then $f$ is an element of $s_d(R)$, and the kernel $\ker(f)$ agrees with that considered in \Cref{rem:kernel}.
\end{rem}
\subsection{Lannes' \texorpdfstring{$T$}{T}-functor}\label{sec:tfunctor}
In this section we review Lannes' $T$-functor, and some standard properties of it. This section overlaps with \cite[Section 2]{heard_depth}.

We recall that Lannes' $T$-functor $T_E$ is left adjoint to $ - \otimes H_E^*$ on the category of unstable modules, i.e., there is an isomorphism
\[
\Hom_{\cal{U}}(T_EM,N) \cong \Hom_{\cal{U}}(M,H_E^* \otimes N),
\]
for $M,N \in \cal{U}$. Although it is relativity elementary to see that such a functor exists (for example, by the adjoint functor theorem), the following results of Lannes \cite{lannes_ihes} are far more surprising.
\begin{thm}[Lannes]
  The functor $T_E \colon \cal{U} \to \cal{U}$ is exact, and commutes with tensor products. Moreover, it restricts to a functor $T_E \colon \cal{K} \to \cal{K}$.
\end{thm}
For any unstable algebra $R$, we write $T^0_ER$ for the $\F_p$-vector space of degree 0 elements of $T_ER$. By \eqref{eq:unstable} this is a $p$-Boolean algebra, i.e., a commutative, unital, $\F_p$-algebra in which $x^p = x$ for any element $x$.

Given a $\cal{K}$-morphism $f \colon R \to H_E^*$, the adjoint is a map $T_ER \to \F_p$. Since $\F_p$ is concentrated in degree 0, we get a map $T^0_ER \to \F_p$. We can then define
\[
T_E(R;f) = T_ER \otimes_{T^0_ER} \F_p(f),
\]
where $\F_p(f)$ denotes $\F_p$ with the $T^0_ER$-module structure coming from the above map. If $R$ is Noetherian, then the $T$-functor decomposes as a finite direct sum of unstable algebras (see for example the discussion around (2.6) of \cite{heard_depth})
\[
T_E(R) = \bigoplus_{f \in \Hom_{\cal{K}}(R,H_E^*)}T_E(R;f).
\]
The components $T_E(R;f)$ are better behaved than $T_E(R)$ itself, in the sense that if $R$ is connected, then so are the $T_E(R;f)$. If $M \in R-\cU$, then we also define
\[
T_E(M;f) = T_EM \otimes_{T_E^0R} \F_p(f).
\]
The following is \cite[Lemma 3.1]{DwyerWilkerson1992cohomology}.
\begin{lem}\label{lem:dwyer_wilerson3.1}
	Let $(E,f) \in \bV_R$, then the set $\Hom_{\cal{K}}(T_E(R;f),S)$ is naturally isomorphic to the set of $\cal{K}$-maps $g \colon R \to H_E^* \otimes S$ making the diagram
	\[
\begin{tikzcd}
	R \arrow{r}{g} \arrow{d}[swap]{f} & H_E^* \otimes S \arrow{d}{1 \otimes \epsilon_S} \\
	H_E^* \otimes \F_p \arrow{r}[swap]{1 \otimes \xi_S} & H_E^* \otimes S^0
\end{tikzcd}
	\]
	commute, where $\epsilon_S \colon S \to S^0$ is projection onto the degree 0 component, and $\xi_S \colon \F_p \to S^0$ is the unit inclusion.
\end{lem}
Given a morphism $\phi \colon T_E(R;f) \to S$ in $\cal{K}$ as in the previous lemma, we write $\phi^{\#}$ for the corresponding map $R \to H_E^* \otimes S$, and call this the adjoint of $\phi$. Likewise, given a map $g \colon R \to H_E^* \otimes S$ satisfying the conditions of the lemma, we call the corresponding map $T_E(R;f) \to H_E^*$ the adjoint of $g$.

We will need  the following maps, where $\epsilon_E \colon H_E^* \to \F_p$ is the canonical map.
\begin{defn} Let $R$ be a unstable algebra, and $(E,f) \in \bV_R$. We define maps:
\begin{enumerate}
	\item $\eta_{R,(E,f)} \colon R \to H_E^* \otimes T_E(R;f)$ as the adjoint of $\text{id} \colon T_E(R;f) \to T_E(R;f)$.
	\item $\rho_{R,(E,f)} \colon R \to T_E(R;f)$ as the composite map $(\epsilon_{E} \otimes 1 )\circ \eta_{R,(E,f)}$.
	\item $\kappa_{R,(E,f)} \colon T_E(R;f) \to H_E^* \otimes T_E(R;f)$ as the adjoint to the composite
	\[
R \xr{\eta_{R,(E,f)}} H_E^* \otimes T_E(R;f) \xr{\Delta \otimes 1} H_E^* \otimes H_E^* \otimes T_E(R;f).
	\]
\end{enumerate}
\end{defn}
As shown in \cite[Section 1.13]{HennLannesSchwartz1995Localizations} for each $E$, the map $\kappa_{R,(E,f)}$ gives $T_E(R;f)$ the structure of a $H_E^*$-comodule.

Note that any map $g \colon T_E(R;f) \to S$ can be written as $T_E(R;f) \xr{\text{id}} T_E(R;f) \xr{g} S$, and taking adjoints we see that $g^{\#} \colon R \to H_E^* \otimes S$ is isomorphic to the composite $(1 \otimes g) \circ \eta_{R,(E,f)}$. This gives the following, which is the component-wise version of \cite[Lemma 2.3]{heard_depth}.
\begin{lem}\label{lem:commdia}
	For any map $g \colon T_E(R;f) \to S$ the diagram
\[
\begin{tikzcd}
	R \arrow{r}{\rho_{R,(E,f)}} \arrow[swap]{d}{g^{\#}} & T_E(R;f) \arrow{d}{g} \\
	H_E^* \otimes S \arrow[swap]{r}{\epsilon_E \otimes 1} & S,
\end{tikzcd}
\]
commutes.
\end{lem}
\begin{proof}
	As noted, $g^{\#}$ factors as the composite $(1 \otimes g)\circ \eta_{R,(E,f)}$. It follows that
\[
\begin{split}
(\epsilon_E \otimes 1)\circ g^{\#} &\cong (\epsilon_E \otimes 1)\circ (1 \otimes g)\circ \eta_{R,(E,f)}\\
&\cong   g \circ (\epsilon_E \otimes 1)\circ \eta_{R,(E,f)}\\
& \cong g \circ \rho_{R,(E,f)}
\end{split}
\]
as required.
\end{proof}
The next result follows immediately from \Cref{lem:commdia} and the definitions of the maps involved.
\begin{cor}\label{lem:comm}
	We have $ \kappa_{R,(E,f)} \circ \rho_{R,(E,f)} \cong \eta_{R,(E,f)}$.
\end{cor}
The assignment $(E,f) \mapsto T_E(M;f)$ extends to a functor $\bV_R \to R-\cal{U}$; in fact, if $R$ is Noetherian, and $M \in R_{fg}-\cal{U}$, then using \cite[Corollary 1.12]{Henn1996Commutative} we even obtain a functor $\bA_R \to R_{fg}-\cal{U}$. Given a morphism $\alpha \colon (E,f) \to (V,g) \in \bA_R$, we will write $T_{\alpha}(g) \colon T_E(R;f) \to T_V(R;g)$ for the induced map. By naturality, we deduce the following.
\begin{lem}\label{lem:central_comm}
	For any morphism $\alpha \colon (E,f) \to (V,g) \in \bA_R$, there is a commutative diagram
	\[
	\begin{tikzcd}
R \arrow[r, "{\eta_{R,(E,f)}}"] \arrow[d, "{\eta_{R,(V,g)}}"'] & H_E^* \otimes T_E(R;f) \arrow[d, "1 \otimes T_{\alpha}(g)"] \\
H_V^* \otimes T_V(R;g) \arrow[r, "\alpha^* \otimes 1"']    & H_E^* \otimes T_V(R;g)
\end{tikzcd}
	\]
\end{lem}
	Finally, we have the useful result \cite[Lemma 4.8]{HennLannesSchwartz1995Localizations}.
  \begin{lem}[Henn--Lannes--Schwartz]\label{lem:hls_lemma}
    Let $R$ be a Noetherian unstable algebra, $M \in R_{fg}-\cU$, and $\alpha \colon E \to E'$ an epimorphism. Then for each $f \in \Hom_{\cK}(R,H_{E'}^*)$ the map $\alpha$ induces an isomorphism
    \[
T_E(M;\alpha^*f) \xr{\simeq} T_{E'}(M;f).
    \]
  \end{lem}
  Finally, it is worth pointing out the following result, which is a consequence of \cite[Proposition 2.1.3]{lannes_ihes}.
  \begin{lem}\label{lem:even_degrees_preserve}
     If $R$ is an unstable algebra concentrated in even degrees, then so are $T_E(R)$ and $T_E(R;f)$ for any $(E,f) \in \bA_R$.
   \end{lem} 
	\begin{ex}\label{ex:groups}
		A fundamental computation is that of $T_E(H_G^*)$ where $G$ is a compact Lie group, due to Lannes \cite{lannes_unpublished,lannes_ihes}. More specifically, let $E < G$ be an elementary abelian $p$-subgroup, with induced map $\res_{G,E}^* \colon H_G^* \to H_E^*$. The multiplication map $E \times C_E(G) \to G$ induces a morphism $H_G^* \to H_E^* \otimes H_{C_E{(G)}}^*$. The adjoint to this gives rise to an isomorphism
		\[
T_E(H_G^*;\res_{G,E}^*) \cong H^*_{C_G(E)}.
		\]
		Moreover, the maps $\eta_{H_G^*,(E,\res_{G,E}^*)}$, $\rho_{H_G^*,(E,\res_{G,E}^*)}$ and $\kappa_{H_G^*,(E,\res_{G,E}^*)}$ are the maps induced on cohomology by the obvious maps
		\[
		\begin{split}
&E \times C_G(E) \to G \\
&C_G(E) \to G \\
&E \times C_G(E) \to C_G(E).
\end{split}
		\]
		Note that the claims of \Cref{lem:comm,lem:hls_lemma} are clear in this case.

		It follows that $T_E(R;f)$ plays the role of the `centralizer' of the pair $(E,f) \in \bA_R$. We investigate this analogy further in the following sections. \qed
	\end{ex}
  \subsection{The nilpotent filtration of an unstable algebra}\label{sec:nilpotent_filtratio}
In this section, we review Schwartz's nilpotent filtration of the category of unstable modules over the Steenrod algebra, and the associated localization functors of Henn, Lannes, and Schwartz.
We recall that in the previous section we introduced the categories $\cal{U}$ and $\cal{K}$ of unstable modules and unstable algebras over the Steenrod algebra respectively. As noted in the introduction, Schwartz \cite{Schwartz1988La} introduced a natural filtration on $\cal{U}$, known as the nilpotent filtration. We take the following from \cite{HennLannesSchwartz1995Localizations}.
\begin{defn}
Let $M,N$ be unstable modules.
\begin{enumerate}
  \item $M$ is called $n$-nilpotent if and only if every finitely generated submodule admits a filtration such that each filtration quotient is an $n$-fold suspension.
  \item The category $\Nil_n$ is the full subcategory of $\cU$ that contains all $n$-nilpotent modules.
  \item$N$ is called $\Nil_n$-reduced if and only if $\Hom_{\cU}(M,N) = 0$ for all $M \in \Nil_n$, and $\Nil_n$-closed if and only if $\Ext_{\cal{U}}^i(M,N) = 0$ for $i = 0,1$ and all $n$-nilpotent modules $M$.
\end{enumerate}
\end{defn}
Further equivalent conditions for $n$-nilpotent modules, and more information about the nilpotent filtration can be found in \cite[Chapter 6]{schwartz_book}, or the fundamental paper of Henn, Lannes, and Schwartz \cite{HennLannesSchwartz1995Localizations}.

The nilpotent filtration leads to the following definition \cite[Def.~3.5]{HennLannesSchwartz1995Localizations}.
\begin{defn}\label{def:d0}
  Let $M$ be an unstable $\cA$-module, then the topological nilpotence degree of $M$ is
  \[
d_0M \coloneqq \inf \{ k \in \mathbb{N} | M \text{ is } \Nil_{k+1}\text{-reduced} \}.
  \]
\end{defn}
We note that if $R$ is Noetherian, and $M \in R_{fg}-\cU$, then $d_0(M)$ is finite \cite[Theorem 4.3]{HennLannesSchwartz1995Localizations}. In particular, $d_0(R)$ itself is finite.

There are a number of alternative characterizations of the number $d_0$. For example, the subcategories $\Nil_n$ are localizing, and the general theory of localization in abelian categories implies there exists a functor $L_n \colon \cU \to \cU$, and a natural transformation $\lambda_n \colon 1_{\cU} \to L_n$ such that $L_nM$ is $\Nil_n$-closed, and $\lambda_n$ has $n$-nilpotent kernel and cokernel. In this case, we have
\[
d_0M = \inf \{ k \in \mathbb{N} | \lambda_{k+1}M \text{ is a monomorphism}\}.
\]
Further equivalent characterizations can be found in \cite[Definition 3.11]{Kuhn2007Primitives}.  One particular result of interest for us is the following, which is a direct consequence of \cite[Theorem 4.9]{HennLannesSchwartz1995Localizations}.
\begin{prop}\label{prop:injection}
  Let $R$ be a Noetherian unstable algebra, and $M \in R_{fg}-\cU$, then for $n\ge d_0(M)$ there is a monomorphism in $R_{fg}-\cU$:
  \[
  \begin{tikzcd}
  \phi_M \colon M \arrow{r} & \displaystyle\prod_{(E,f) \in \bA_R} H_E^* \otimes T_E(M;f)^{\le n}.
  \end{tikzcd}
  \]
  induced by the product of the maps $\eta_{M,(E,f)}$.
\end{prop}
Here we write $K^{\le n}$ for the quotient of a graded module $K$ by all elements of degree greater than $n$. Note that if $K$ is an unstable module, then so is the quotient $K^{\le n}$.

We also have the following properties of $d_0$, which are a combination of \cite[Proposition 3.6]{HennLannesSchwartz1995Localizations} and \cite[Proposition 3.12]{Kuhn2007Primitives}.
\begin{prop}\label{prop:dproperties}
Let $M$ be an unstable module.
  \begin{enumerate}
    \item If $M$ is concentrated in finitely many degrees, then $d_0(M) \le n$, where $n$ is the top degree in which $M$ is non-zero.
    \item Let $0 \to M' \to M \to M''$ be an exact sequence in $\cU$, then $d_0M' \le d_0M$.
    \item Let $0 \to M' \to M \to M'' \to 0$ be an exact sequence in $\cU$, then $d_0(M) \le \max\{d_0(M'),d_0(M'')\}$.
    \item $d_0(M \otimes M') = d_0(M) + d_0(M')$.
    \item $d_0(T_EM) = d_0(M)$.
    \item If $M \ne 0$, then $d_0(\Sigma^nM) = d_0(M) + n$.
  \end{enumerate}
\end{prop}

The topological nilpotence degree of a Noetherian unstable algebra $R$ is related to algebraic nilpotence in the following way, compare \cite[Corollary 2.6]{Kuhn2013Nilpotence}.
\begin{lem}\label{lem:algebraic_nilpotence}
  Let $R$ be a connected Noetherian unstable algebra, and define $t$ to be $d_0(R)$ for $p = 2$, or $d_0(R) + \dim(R)$ for $p$ odd. Then $t$ is the maximal integer $d$ such that $\rad(R)^d \ne 0$. In particular, for $s > t$, the product of any $s$ nilpotent elements in $R$ is zero.
\end{lem}
\begin{proof}
  Let $d^{\alg}(R)$ be the maximal $d$ such that $\rad(R)^d \ne 0$, so that our claim is $d^{\alg}(R) \le t$. It is clear that
  \[
d^{\alg}(H_E^* \otimes T_E(R;f)^{\le d}) \le \begin{cases}
  d^{\alg}(T_E(R;f)^{\le d}) & \text{ if } p = 2\\
  d^{\alg}(T_E(R;f)^{\le d}) + \rank(E) & \text{ if } p > 2.\end{cases}
  \]
  It then follows from \Cref{prop:injection} that
  \[
d^{\alg}(R) \le \begin{cases}
  \max\limits_{(E,f) \in \bA_R}\{d^{\alg}(T_E(R;f)^{\le d_0(R)})\} \le d_0(R) & \text{ if } p = 2\\
\max\limits_{(E,f) \in \bA_R}\{d^{\alg}(T_E(R;f)^{\le d_0(R)}) +\rank(E) \} \le d_0(R) +\dim(R) & \text{ if } p > 2.
\end{cases}
  \]
  Here we have used that $\rank(E) \le \dim(R)$ for each $(E,f) \in \bA_R$, see \Cref{prop:rector_props}(2). It follows that $d^{\alg}(R) \le t$ as claimed.
\end{proof}
\begin{rem}[The case of an odd prime]\label{rem:odd_prime}
  The cohomology of elementary abelian $p$-groups (\Cref{ex:elemn_abelian}) shows already one significant difference between working at $p = 2$ or working at an odd prime, namely the presence of the exterior classes. Many of the fundamental results of unstable algebras therefore have slightly different forms in the case of odd primes. One way to deal with these problems is to work with the full subcategory  $\cU' \subseteq \cU$ consisting of unstable modules which are non-trivial only in even degrees. There is an obvious forgetful functor $\cal{O} \colon \cU' \to \cU$ which has a right adjoint $\widetilde{ \cal{O}} \colon \cU \to \cU'$, see \cite{LannesZarati1986Sur} for example, which is the largest submodule of $M$ that is concentrated in even degrees, or even more explicitly
  \[
\widetilde{O}M = \bigcap_{\theta \in \cal{A}} \ker(\beta \theta \colon M^{\text{ev}} \to M).
  \] Similarly, we have the category $\cal{K}'$ of unstable algebras concentrated in even degrees. At certain points it will be convenient for us to assume that our unstable algebra  comes from $\cal{K'}$ (considered naturally as an object in $\cal{K}$) when $p$ is odd.

\end{rem}
  \section{The center of a Noetherian unstable algebra}
  In this section, following Dwyer and Wilkerson, we study  central objects of a Noetherian unstable algebra with respect to the objects of $\bA_R$. The main new result, given here as \Cref{thm:max_central}, is that up to isomorphism there is a maximal such element with respect to a natural poset structure on $\bA_R$. We also prove that for each central object $(E,f) \in \bA_R$, the unstable algebra $R$ naturally obtains the structure of a $H_E^*$-comodule, which will be crucial for the calculation of $d_0(R)$.
	\subsection{Central objects of a Noetherian unstable algebra}
	Throughout this section we assume that $R$ is a connected Noetherian unstable algebra. It would suffice to assume that the module of indecomposables $Q(R)$ is locally finite, i.e., every element of $Q(R)$ is contained in a finite $\cA$-submodule, however we have no need for this greater generality.

  We observe from \Cref{ex:groups} that if $E \le G$ is a central elementary abelian $p$-subgroup of a compact Lie group, then the map $\rho_{H_G^*,(E,\res_{G,E}^*)} \colon H_G^* \to T_E(H_G^*;\res_{G,E}^*)$ is an isomorphism. Based on  this is natural to make the following definition.
	\begin{defn}[Dwyer--Wilkerson]
		Let $R$ be a connected Noetherian unstable algebra, then a pair $(E,f) \in \bA_R$ is called central if $\rho_{R,(E,f)} \colon R \to T_E(R;f)$ is an isomorphism.
	\end{defn}
As noted, given a central elementary abelian $p$-subgroup $E$ of a compact Lie group $G$, the pair $(E,\res_{G,E}^*)$ is then central inside $\bA_{H_G^*}$. We will see later the converse is true if $G$ is a finite $p$-group, but not in general.

A useful criteria for recognizing central objects is given in \cite[Proposition 3.4]{DwyerWilkerson1992cohomology}.
  \begin{prop}[Dwyer--Wilkerson]\label{prop:central_char}
    A pair $(E,f) \in \bA_R$ is central if and only if there exists a $\cK$-map $R \to H_E^* \otimes R$ which, when composed with the projections $H_E^* \otimes R \to R$ and $H_E^* \otimes R \to H_E^*$ gives, respectively, the identity map of $R$ and the map $f$.
  \end{prop}

We now prove some basic facts about central objects, all of which are analogous to standard statements about central subgroups of groups.  We begin with the following result, which is an algebraic analog of the fact that if $E$ is an elementary abelian $p$-subgroup of a group $G$, then $E$ is always a central subgroup of $C_G(E)$.
\begin{prop}\label{prop:central_factor}
    Given $(E,f) \in \bA_R$, there is a $\cal{K}$-map $h \colon T_E(R;f) \to H_E^*$ factoring the map $f \colon R \to H_E^*$. Moreover, the pair $(E,h)$ is central in $\bA_{T_E(R;f)}$.
  \end{prop}
  \begin{proof}
We define the map $h$ as the composite
\[
\xymatrix@C=4em{ T_E(R;f) \ar[r]^-{\kappa_{R,(E,f)}} & H_E^* \otimes T_E(R;f) \ar[r]^-{1 \otimes \epsilon_{T_E(R;f)}} & H_E^*}
\]
To see that this factors the map $f$, note that
\[
\begin{split}
f & \cong(1 \otimes \epsilon_{T_E(R;f)}) \circ \eta_{R,(E,f)} \\
&\cong (1 \otimes \epsilon_{T_E(R;f)}) \circ \kappa_{R,(E,f)} \circ \rho_{R,(E,f)} \\
&= h \circ \rho_{R,(E,f)},
\end{split}
\]
where we have used \Cref{lem:comm}. Finally, because $f$ and $\rho_{R,(E,f)}$ are finite morphisms (the latter by \cite[Corollary 1.12]{Henn1996Commutative}, for example), so is $h$. The composite $(\epsilon_{E} \otimes 1) \circ \kappa_{R,(E,f)} \colon T_E(R;f) \to T_E(R;f)$ is the identity, and therefore $(E,h) \in \bA_{T_E(R;f)}$ is central by \Cref{prop:central_char}.
\end{proof}
\begin{rem}
  The morphism $h \colon T_E(R;f) \to H_E^*$ is in fact a morphism of $H_E^*$-comodules. In fact, unwinding the definitions of the maps involved, this is nothing other than the statement of coassociativity for the $H_E^*$-comodule $T_E(R;f$). 
\end{rem}
We have the following behavior with respect to tensor products.
\begin{lem}\label{lem:tp_center}
  Suppose $R_1$ and $R_2$ are Noetherian connected unstable algebras, and $(E_i,f_i) \in \bA_{R_i}$ is central for $i=1,2$, then $(E_1 \oplus E_2,f)$ is central in $\bA_{R_1 \otimes R_2}$, where $f \colon R_1 \otimes R_2 \to H_{E_1 \oplus E_2}^*$ is the composite $R_1 \otimes R_2 \xr{f_1 \otimes f_2} H_{E_1}^* \otimes H_{E_2}^* \cong H_{E_1 \oplus E_2}^*$.
\end{lem}
\begin{proof}
  This is an almost immediate consequence of the fact that the $T$-functor commutes with tensor-products; there is a natural isomorphism $T_{E_1 \oplus E_2}(R_1 \otimes R_2;f) \cong T_{E_1}(R_1;f_1) \otimes  T_E(R_2;f_2)$, and under this isomorphism $\rho_{R_1 \otimes R_2,(E_1 \oplus E_2,f)}$ corresponds to $\rho_{R_1,(E_1,f_1)} \otimes \rho_{R_2,(E_2,f_2)}$. Alternatively, if $\phi_i \colon R \to H_{E_i}^* \otimes R_i$ is the $\cal{K}$-map arising via \Cref{prop:central_char}, then the $\cK$-map $\phi_1 \otimes \phi_2 \colon R_1 \otimes R_2 \to H_{E_1 \oplus E_2}^* \otimes R_1 \otimes R_2$ satisfies the conditions of \Cref{prop:central_char}, and shows that $(E_1 \oplus E_2,f)$ is central.
\end{proof}

The next two lemmas are due to Dwyer--Wilkerson \cite[Lemma 4.5 and Lemma 4.6]{DwyerWilkerson1992cohomology}; the first is an immediate consequence of \Cref{prop:central_char}.
\begin{lem}[Dwyer--Wilkerson]\label{lem:subgroups_of_central}
  Let $(C,g)$ be in $\bA_R$, and assume that $(C,g)$ is central. If $C'$ is a subgroup of $C$, then $(C',g')$ is central where $g \colon R \xr{g} H_C^* \xr{\iota^*} H_{C'}^*$, and $\iota \colon C' \to C$ is the inclusion.
\end{lem}
\begin{rem}\label{rem:dw_sum_construction}
  Let $(E,f) \in \bV_R$ and $(C,g) \in \bA_R$ with $(C,g)$ central. Observe that, by adjunction, a map $R \to H_E^* \otimes H_C^*$ can be specified by giving a $\cal{K}$-map $T_C(R;g) \to H_E^*$, or equivalently, using that $(C,g)$ is central, a $\cal{K}$-map $R \to H_E^*$. We denote by $f \boxplus g\colon R \to H_E^* \otimes H_C^*$ the map corresponding to $f$, so that $(E \oplus C, f \boxplus g) \in \bV_R$. Dwyer and Wilkerson show the following. 
\end{rem}
\begin{lem}[Dwyer--Wilkerson]\label{lem:sumuniqueness}
	Let $(E,f) \in \bV_R$ and $(C,g) \in \bA_R$ with $(C,g)$ central. Then $(E \oplus C,f \boxplus g)$ is the unique pair in $\bV_R$ which restricts to $f$ (resp.~$g$) along the summand inclusion $E \to E \oplus C$ (resp.~$C \to E \oplus C$).
\end{lem}

We observe that it is not necessarily the case that $(E \oplus C,f \boxplus g) \in \bA_R$, i.e., the map $f \boxplus g \colon R \to H_{E \oplus C}^*$ is not necessarily finite.  As discussed in \Cref{rem:kernel}, there is a functor $\rec \colon \bV_R \to \bA_R$, given by $(V,j) \mapsto (V/\ker(j),\tilde j)$ and applying this to the construction in \Cref{lem:sumuniqueness} leads to the following definition.
\begin{defn}\label{defn:boxplus}
	Let $(E,f)$ and $(C,g)$ be objects of $\bA_R$, and assume that $(C,g)$ is central, then we let $(E \circ C,\sigma(f,g)) \coloneqq \rec(E \oplus C, f \boxplus g)$ be the object in $\bA_R$ corresponding to $(E \oplus C, f \boxplus g) \in \bV_R$, i.e., $E \circ C = E \oplus C/\ker(f \boxplus g)$.
\end{defn}
As a diagram, we can represent this as
\[
\begin{tikzcd}
H_E^* & R \arrow[r, "g"] \arrow[l, "f"'] \arrow[d, "{\sigma(f,g)}"] & H_C^* \\
      & H_{E \circ C}^* \arrow[d, "q^*"]                          &       \\
      & H^*_{E \oplus C} \arrow[luu] \arrow[ruu]                    &
\end{tikzcd}
\]
where $q^*$ is induced by $q \colon E \oplus C \to E \circ C$, and the composite $q^* \circ \sigma(f,g) \cong f \boxplus g$. Note that the natural maps $E \to E \circ C$, and $C \to E \circ C$, induce maps $T_E(R;f) \to T_{E \circ C}(R;\sigma(f,g))$ and $T_C(R;f) \to T_{E \circ C}(R;\sigma(f,g))$.

\subsection{The poset of central objects}
Observe that the category $\bA_R$ has the property that every endomorphism is an isomorphism. Such a category is called an $EI$-category (see \cite{Luck1989Transformation}), and the set of isomorphism classes of objects is partially ordered by the relation
\[
[(E,f)] \le [(V,g)] \quad \text{if} \quad \Hom_{\bA_R}((E,f),(V,g)) \ne \emptyset.
\]
Recall that this implies that there exists a monomorphism $\iota \colon E \hookrightarrow V$ such that the diagram
\[
\begin{tikzcd}
      & R \arrow[ld, "f"'] \arrow[rd, "g"] &                  \\
H_E^* &                                    & H_V^* \arrow[ll,"\iota^*"]
\end{tikzcd}
\]
We will write $(E,f) \subseteq (V,g)$ if $[(E,f)] \le (V,g)]$.

Consider the full subcategory $\bA_R^{\mathrm{central}} \subset \bA_R$ consisting of the central objects. This inherits the partial order from $\bA_R$. We shall show that,  with respect to this partial order, $\bA_R^{\mathrm{central}}$ has, up to isomorphism, a unique maximal element, i.e., there is, up to isomorphism, a unique maximal central object in $\bA_R$. To do this, we briefly recall the definition of an under category.
\begin{defn}
	Given $(E,f) \in \bA_R$, the under category $(E,f) \downarrow \bA_R$ is the category with objects pairs $(\alpha,(V,g))$ where $\alpha \colon (E,f) \to (V,g)$ is a morphism in $\bA_R$, and a morphism $(\alpha,(V,g)) \to (\alpha',(W,h))$ is a morphism $f \colon (V,g) \to (W,h)$ in $\bA_R$ such that the diagram
	\[
	\begin{tikzcd}
& (E,f) \arrow[swap]{dl}{\alpha} \arrow{dr}{\alpha'}&  \\
(V,g) \arrow[swap]{rr}{f} && (W,h)
\end{tikzcd}
	\]
	commutes.
	\end{defn}
	A crucial observation is the following, which is shown in the proof of Proposition 4.10 of \cite{DwyerWilkerson1992cohomology}.
\begin{prop}[Dwyer--Wilkerson]\label{prop:dwfunctor}
	Let $(C,g)$ be central, then for any $(E,f) \in \bA_R$ the assignment $(E,f) \mapsto (E \circ C,\sigma(f,g))$ defines a functor $\sigma \colon \bA_R \to (C,g) \downarrow \bA_R$. Moreover, the natural map
	\[
\iota \colon T_E(R;f) \to T_{E \circ C}(R;\sigma(f,g))
	\]
	induced by $E \to E \circ C$ is an isomorphism.
\end{prop}

\begin{cor}\label{cor:central}
	If $(C,g)$ and $(E,f)$ are central in $\bA_R$, then so is $(E \circ C,\sigma(f,g))$.
\end{cor}
\begin{proof}
	By the previous proposition $\iota \colon T_E(R;f) \to T_{E \circ C}(R;\sigma(f,g))$ is an isomorphism. Centrality of $T_E(R;f)$ implies that $\rho_{R,(E,f)} \colon R \to T_E(R;f)$ is an isomorphism, and hence so is the composite $\iota \circ \rho_{R,(E,f)}$. Observe that $\rho_{R,\sigma(f,g)} \cong \iota \circ \rho_{R,(E,f)}$. This is clear because the map $\{ e \} \to E \circ C$ factors through $\{ e \} \to E$. It follows that $\rho_{R,\sigma(f,g))} \colon R \to T_{E \circ C}(R;\sigma(f,g))$ is an isomorphism, and $(E \circ C, \sigma(f,g)) \in \bA_R$ is central.
\end{proof}
\begin{thm}\label{thm:max_central}
With respect to the poset structure above, there exists a unique (up to isomorphism) maximal central element $(C,g) \in \bA_R$.
\end{thm}
\begin{proof}
By \Cref{prop:rector_props}(1) there are only finitely many isomorphism classes of objects in $\bA_R$ and hence $\bA_R^{\mathrm{central}}$. It follows that there exist maximal isomorphism classes of central objects. We now show that there is a unique such isomorphism class. To that end, suppose we are given two central objects $(E,f)$ and $(V,g)$ in $\bA_R$. By \Cref{prop:dwfunctor} the pair $(E \circ V,\sigma(f,g)) \in (E,f) \downarrow \bA_R$ and by symmetry $(E \circ V,\sigma(f,g)) \in (V,g) \downarrow \bA_R$. In particular, we have $(E,f) \subseteq (E \circ V,\sigma(f,g)) \supseteq (V,g)$. Moreover, by \Cref{cor:central} $(E \circ V,\sigma(f,g))$ is central in $\bA_R$. This implies the result.
\end{proof}
\begin{defn}\label{defn:center}
	Let $R$ be a Noetherian unstable algebra, then the center $(C,g) \in \bA_R$ is a choice of representative for the isomorphism class of the maximal central object with respect to the poset structure on $\bA_R$. We say that rank of the center is the rank of the elementary abelian $p$-group $C$.
\end{defn}
\begin{ex}[Mislin]\label{ex:centers3}
	The following example, due to Mislin \cite{Mislin1992Cohomologically}, shows that if $R = H_G^*$ for a compact Lie group, then the center of $H_G^*$ need not agree with the maximal central elementary abelian $p$-subgroup of $G$. Let $G = \Sigma_3$ and work at the prime 2, then the inclusion $C_2 \to \Sigma_3$ of a 2-Sylow subgroup induces an isomorphism $H_{\Sigma_3}^* \cong H_{C_2}^*$. Moreover, we have $T_{C_2}(H_{\Sigma_3}^*;\res^*_{\Sigma_3,C_2}) \cong H_{C_{\Sigma_3}(C_2)}^* \cong  H_{C_2}^*$. In particular, the map $H_{\Sigma_3}^* \to T_{C_2}(H_{\Sigma_3}^*;\res^*_{\Sigma_3,C_2})$ is an isomorphism. This shows that the pair $(C_2,\res_{\Sigma_3,C_2}^*)$ is central in $H_{\Sigma_3}^*$. In fact, $(C_2,\res_{\Sigma_3,C_2}^*)$ is the center of $H_{\Sigma_3}^*$. Note that $\Sigma_3$ actually has 3 conjugate elementary abelian subgroups of order $2$, and that $\Sigma_3$ has trivial group theoretic center.

  Mislin shows more generally that, at the prime $p$, the center of ${H_G^*}$ is equal to the maximal elementary abelian $p$-subgroup of the center of $G/\cal{O}_{p'}(G)$, where $\cal{O}_{p'}(G)$ denote the largest $p'$-normal subgroup of $G$ for $p'$ a prime not equal to $p$. Thus, if $G$ is a finite $p$-group, then the center of $H_G^*$ is equal to the maximal central elementary abelian $p$-subgroup of $G$, but not in general. In particular, if $G$ is a finite $p$-group, then the center of $H_G^*$ is always non-trivial.   \qed
\end{ex}
\begin{ex}[Modular invariant theory]
  Let $V$ be a finite-dimensional $\F_p$-vector space, $G$ a finite group such that $p$ divides the order of $G$, and $\rho \colon G \to GL_n(V)$ a faithful modular representation. We let $\F[V]$ denote the graded algebra of polynomial functions on $V$ with generators in degree 2, which is a graded $\F_p$-algebra with a unique action of the Steenrod algebra. These operations commute with the action of $G$, and define an action on $\F[V]^G$, see \cite[Section 5]{DwyerWilkerson1998Kahler} or \cite[Chapter 8]{NeuselSmith2002Invariant} for more details. Moreover, $\F[V]^G$ is a finitely-generated $\F_p$-algebra, see for example, \cite[Corollary 2.1.5]{NeuselSmith2002Invariant}. In other words, $\F[V]^G$ defines a connected Noetherian unstable algebra.

  Let $i \colon U \hookrightarrow V$ be the inclusion of an $\F_p$-linear subspace $U$, then we can define a morphism $f_U \colon \F[V]^G \hookrightarrow \F[V] \xr{i^*} \F[U] \to H_U^*$, which is a $\cal{K}$-map. In fact every morphism in $\bA_{\F[V]^G}$ arises this way; Rector's category $\bA_{\F[V]^G}$ is equivalent to the category of pairs $(U,f_U)$ such that $H_U^*$ is a finite $\F[V]^G$-module via $f_U$. This can be deduced from the proof of Theorem 1.1 of \cite{DwyerWilkerson1998Kahler}. Moreover, Dwyer and Wilkerson also prove that
  \[
T_U(\F[V]^G;f_U) \cong \F[V]^{G_U}
  \]
  where $G_U \le G$ is the pointwise stabilizer of $U$, i.e., $G_U = \{ g \in G \mid g \cdot u = u \}$.

  Let $V^G = \{v \in V \mid g \cdot v = v \}$ denote the $G$-invariant subspace $V^G \subseteq V$. It follows from the discussion above that the center of $F[V]^G$ is $(V^G,f_{V^G})$. Note that because the representation is faithful, we have $\dim_{\F_p}(V^G) < \dim_{\F_p}(V)$, giving an upper bound for the rank of the center.  Moreover, if $G$ is a $p$-group, then $V^G \ne 0$, so that the center is non-trivial in this case. This is a direct analogue of the fact that (non-trivial) $p$-groups have non-trivial centers, and hence, in light of the previous example, that the unstable algebra $H_G^*$ always has non-trivial center when $G$ is a $p$-group.  \qed
\end{ex}
\begin{ex}(Noetherian $H$-spaces)\label{ex:hspaces}
  Suppose $X$ is a connected $H$-space with Noetherian mod $p$-cohomology. The mod $p$ cohomology is given by
\begin{equation}\label{eq:hspace}
H^*(X) \cong \F_2[x_1,\ldots,x_r] \otimes \frac{\F_2[y_1,\ldots,y_s]}{(y_1^{2^{a_1}},\ldots,y_s^{2^{a_s}})},
\end{equation}
when $p = 2 $ see, for example, \cite[Equation (5)]{BrotoCrespo1999spaces} and
\begin{equation}\label{eq:hspace_odd}
H^*(X;\F_p) \cong \frac{\F_p[y_1,\ldots,y_s]}{(y_1^2,\ldots,y_s^2)} \otimes \F_p[\beta y_1,\ldots, \beta y_k,x_{k+1},\ldots,x_n] \otimes \frac{\F_p[z_1,\ldots,z_t]}{(z_1^{p^{a_1}},\ldots,z_t^{p^{a_t}})},
\end{equation}
when $p$ is odd, where $\beta$ is the Bockstein \cite[Corollary 2.7]{Crespo}. Note that in both cases the generators cannot take arbitrary degrees, see \cite[Theorem 1.6]{BrotoCrespo1999spaces} when $p = 2$ and \cite[Proposition 2.8]{Crespo} when $p$ is odd.

We claim that the center of $H^*(X)$ has rank equal to the Krull dimension of $H^*(X)$ (note that this is the maximal possible rank by \Cref{prop:rector_props}). Indeed, in both cases there exists a map of Hopf algebras $f \colon H^*(X) \to H_E^*$ where $E$ is elementary abelian of rank equal to the Krull dimension of $H^*(X)$ - when $p = 2$ this is \cite[Theorem 2.2]{BrotoCrespo1999spaces} and when $p$ is odd this is \cite[Theorem 2.6]{Crespo}. This map is in fact the localization away from $\Nil_1$ of $H^*(X)$ and hence these morphisms are finite \cite[Corollary 4.10]{HennLannesSchwartz1995Localizations}. That they are central then follows from \cite[Theorem 3.2 and Lemma 4.5]{DwyerWilkerson1990Spaces}, see also the proof of Theorem 2.3 of \cite{BrotoCrespo1999spaces}.
  \qed
\end{ex}

We now show that the center is well-behaved under tensor products of unstable algebras.
\begin{lem}\label{lem:center_of_tensor}
  Suppose $R_1$ and $R_2$ are Noetherian connected unstable algebras with center $(C_i,g_i) \in \bA_{R_i}$  for $i=1,2$, then $(C_1 \oplus C_2,g)$ is the center of $\bA_{R_1 \otimes R_2}$, where $g \colon R_1 \otimes R_2 \to H_{C_1 \oplus C_2}^*$ is the composite $C_1 \otimes C_2 \xr{f_1 \otimes f_2} H_{C_1}^* \otimes H_{C_2}^* \cong H_{C_1 \oplus C_2}^*$.
\end{lem}
\begin{proof}
  By \Cref{lem:tp_center} $(C_1 \oplus C_2,g) \in \bA_{R_1 \otimes R_2}$ is central, so it remains to show that it is maximal. Suppose then that $(C_1 \oplus C_2,g) \subsetneq (V,j)$ for some $(V,j) \in \bA_{R_1 \otimes R_2}$. In particular, we have a monomorphism $\alpha \colon C_1 \oplus C_2 \to V$. The composition $C_1 \xhookrightarrow{\iota_1} C_1 \oplus C_2 \xr{\alpha} V$ is a monomorphism, and so we may identify $C_1$ with $\iota_1\alpha(C_1) \subset V$ and choose a complement $V_1$ of $C_1$, and similar for $C_2$. We can then produce $(V_i,j_i) \in \bA_{R_i}$ for $i = 1,2$ such that $(C_1,g_1) \subseteq (V_1,j_1)$ and $(C_2,g_2) \subseteq (V_2,j_2)$. From centrality of $(V,j)$ we easily see that $(V_i,j_i) \in \bA_{R_i}$ is central. We now observe that either $(C_1,g_1) \subsetneq (V_1,j_1)$ or $(C_2,g_2) \subsetneq (V_2,j_2)$ for otherwise we could not have $(C_1 \oplus C_2,g) \subsetneq (V,j)$. This contradicts the assumption that $(C_i,g_i)$ is the center of $R_i$, and so $(C_1 \oplus C_2,g)$ is maximal, and hence is the center of $\bA_{R_1 \otimes R_2}$.
\end{proof}
\begin{ex}
  We now consider three unstable algebras which we shall see have trivial center, as suggested by the referee.
  \begin{enumerate}
    \item Consider the square-zero extension $R = \F_2[x] \oplus \Sigma \F_2$ where $|x| =1$, or equivalently $R = H_{\Z/2}^* \oplus \Sigma \F_2$.  Let $f \colon R \to H_{\Z/2}^*$ denote the projection map. Since $T_{\Z/2}(-;f)$ is exact, we must compute $T_{\Z/2}(\F_2[x];f)$ and $T_{\Z/2}(\Sigma \F_2;f)$. Both are well-known from Lannes' computations: $T_{\Z/2}(\F_2[x],f) \cong \F_2[x]$, while $T_{\Z/2}(\Sigma \F_2;f) \cong \Sigma T_{\Z/2}(\F_2;f)$ is trivial. In this case, the map $\rho_{R,(\Z/2,f)} \colon R \to T_{\Z/2}(R;f)$ is not an isomorphism, so $(\Z/2,f) \in \bA_R$ is not central. This is in fact the only non-trivial element in $\bA_R$, so we conclude that the center is trivial.

    An alternative way to see that the center must be trivial is to note that the depth (see \Cref{sec:appendix}) of $R$ is 0. By \Cref{cor:duflot} the depth of $R$ is always at least the rank of the center; in particular, if $\depth(R) = 0$, then the center must be trivial.
    \item Let $R$ denote the sub-algebra of $\F_2[x]$ omitting the class of degree 1. In this case, the category $\bA_R$ contains only the trivial subgroup (because $R$ has no class in degree 1), and therefore $R$ has trivial center.
    \item Let $\overline \F_2[x]$ denote the augmentation ideal of $\F_2[x]$, then consider the unstable algebra $R = \F_2 \oplus \overline{ \F_2[x]}^{\oplus 2}$. There are two maps $R \to H_{\Z/2}^*$ given by the projection onto $\overline{\F_2[x]}$ and the inclusion map. A similar argument to (1) shows that these maps cannot be central, and so $R$ has trivial center.
  \end{enumerate}
\end{ex}
\subsection{Hopf algebras and comodules}
One of the key properties of $H_G^*$ used by Kuhn is that for a central elementary abelian subgroup $C$, $H_G^*$ is a $H_C^*$-comodule, and moreover the restriction map $H_G^* \to H_C^*$ is a morphism of $H_C^*$-comodules. A similar result occurs for general unstable algebras.
\begin{prop}\label{prop:restriction_comodule}
	Let $R$ be a connected Noetherian unstable algebra and $(V,j) \in \bA_R^{\mathrm{central}}$, then $R$ is a $H_V^*$-comodule, and $j \colon R \to H_V^*$ is a morphism of $H_V^*$-comodules. In particular, the image of $j$ is a sub-Hopf algebra of $H_V^*$.
\end{prop}
\begin{proof}
We recall that $\kappa_{R,(V,j)} \colon T_V(R;j) \to H_V^* \otimes T_V(R;j)$ makes $T_V(R;j)$ into a $H_V^*$-comodule; since $\rho_{R,(V,j)}$ is central, it follows that $R$ is also a $H_V^*$-comodule.

 That $j \colon R \to H_V^*$ is a morphism of $H_V^*$-comodules follows from a diagram chase. Indeed, the diagram
 \[
\begin{tikzcd}[column sep=0.76in]
R \arrow[r, "\Psi_{R,(V,j)}"] \arrow[d, "\Psi_{R,(V,j)}"']     & H_V^* \otimes R \arrow[r, "1 \otimes \epsilon_R"] \arrow[d, "\Delta \otimes 1"] & H_V^* \arrow[d, "\Delta"] \\
H_V^* \otimes R \arrow[r, "1 \otimes \Psi_{R,(V,j)}"'] & H_V^* \otimes H_V^* \otimes R \arrow[r, "1 \otimes 1 \otimes \epsilon_R"']      & H_V^* \otimes H_V^*.
\end{tikzcd}
\]
commutes, and the top composite is $j$, while the bottom composite is $1 \otimes j$. Here $\Psi_{R,(V,j)}$ denotes the comodule structure map.

To see that the image $K$ of $j$ is a sub-Hopf algebra of $H_V^*$ follows as in the proof of Theorem 1.2 of \cite{BrotoHenn1993Some}. We recall their argument here: The map $j$ is a morphism of $H_V^*$-comodule algebras, and therefore $K$ is a sub-comodule algebra of $H_V^*$. Because $H_V^*$ is a commutative Hopf algebra, the restriction of the diagonal in $H_V^*$ to $K$ has image in $K \otimes K$, and hence $K$ is a sub-Hopf algebra of $H_V^*$.
\end{proof}
We have the following corollary which, as noted in \cite[Remark 1.3]{BrotoHenn1993Some}, follows from the Borel structure theorem \cite[Theorem 7.11]{MilnorMoore1965structure}.
\begin{cor}\label{cor:Borel}
There is a basis $x_1,\ldots,x_c$ for $H_V^1$ such that
\begin{equation}\label{eq:structure}
K = \begin{cases}
  \F_2[x_1^{2^{j_1}},\ldots,x_c^{2^{j_c}}] &\text{ if } p=2\\
  \F_p[y_1^{p^{j_1}},\ldots,y_b^{p^{j_b}},y_{b+1},\ldots,y_c] \otimes \Lambda(x_{b+1},\ldots,x_c) &
  \text{ if } p \text{ is odd,}
\end{cases}
\end{equation}
for some natural numbers $j_1 \ge j_2 \ge \cdots $, and where $y_i = \beta x_i$ for $\beta$ the Bockstein homomorphism.
\end{cor}

Now suppose we are given $(V,j)$ in $\bA^{\mathrm{central}}_R$ and a non-trivial homomorphism $\alpha \colon (V,j) \to (E,f)$ in $\bA_R$; in particular, there is a monomorphism $\alpha \colon V \hookrightarrow E$, and $(E,f)$ need not be central. As discussed previously, $T_V(R;j)$ is a $H_V^*$-comodule via $\kappa_{R,(V,j)}$ and $T_E(R;f)$ is a $H_E^*$-comodule via $\kappa_{R,(E,f)}$. Moreover, if we compose with the coalgebra morphism $\alpha^* \colon H_E^* \to H_V^*$, then $T_E(R;f)$ becomes a $H_V^*$-comodule via $\alpha^* \circ \kappa_{R,(E,f)}$, and moreover $T_{\alpha}(f) \colon T_V(R;j) \to T_E(R;f)$ is a morphism of $H_V^*$-comodules, see the discussion (before passing to components) on the bottom of page 30 of \cite{HennLannesSchwartz1995Localizations}. In particular, the following diagram commutes:
\[
\begin{tikzcd}
T_V(R;j) \arrow[r, "T_{\alpha}(f)"] \arrow[d, "{\kappa_{R,(V,j)}}"'] & T_E(R;f) \arrow[d, "{\alpha^*\circ \kappa_{E,(R,f)}}"] \\
H_V^* \otimes T_V(R;j) \arrow[r, "1 \otimes T_{\alpha}(f)"']         & H_V^* \otimes T_E(R;f)
\end{tikzcd}
\]
 This leads to the following result.
\begin{lem}\label{lem:comodulemorphism}
	Suppose $R$ is a connected Noetherian unstable algebra, $(V,j)$ is central, and that $(V,j) \subseteq (E,f)$ (so that there is a non-trivial homomorphism $\alpha \colon (V,j) \to (E,f)$ in $\bA_R$). With the comodule structures as described above, $\rho_{R,(E,f)} \colon R \to T_E(R,f)$ is a morphism of $H_V^*$-comodules.
\end{lem}
\begin{proof}
	By definition of the comodule structures, we must show that the diagram
	\[
\begin{tikzcd}
R \arrow[d, "\rho_{R,(C,g)}"'] \arrow{d}{\cong} \arrow[r, "\rho_{R,(E,f)}"]                                     & T_E(R;f) \arrow[d,equal]                                      \\
T_V(R;j) \arrow[r, "T_{\alpha}(f)"] \arrow[d, "{\kappa_{R,(V,j)}}"']                         & T_E(R;f) \arrow[d, "{\alpha^\ast\circ \kappa_{R,(E,f)}}"] \\
H_V^* \otimes T_V(R;g) \arrow[r, "1 \otimes T_{\alpha}(f)"]                                 & H_V^* \otimes T_E(R;f)                                   \\
H_V^* \otimes R \arrow[u, "{1 \otimes \rho_{R,(V,j)}}"] \arrow{u}[swap]{\cong} \arrow[r, "1 \otimes \rho_{R,(E,f)}"'] & H_V^* \otimes T_E(R;f) \arrow[u,equal]
\end{tikzcd}
\]
commutes. To see that the top and bottom square commute, we use \Cref{lem:commdia,lem:central_comm} and the definition of $T_{\alpha}(f)$ to see that there are isomorphisms
  \[
  \begin{split}
  T_{\alpha}(f) \circ \rho_{R,(V,j)} & \cong (\epsilon_V \otimes 1) \circ (1 \otimes T_{\alpha}(f)) \circ \eta_{R,(V,j)}\\
  & \cong (\epsilon_V \otimes 1) \circ (\alpha^* \otimes 1) \circ \eta_{R,(E,f)}\\
  & \cong (\epsilon_E \otimes 1) \circ \eta_{R,(E,f)} = \rho_{R,(E,f)}. \\
  \end{split}
  \]
 Finally, the middle square commutes by the fact that $T_{\alpha}(f)$ is a morphism of $H_V^*$-comodules. Thus, the diagram commutes as claimed.
\end{proof}
We also require the following technical lemma. We are grateful to the referee for simplifying the proof. We recall our usual notation: if $(E,f) \in \bA_R$, then there exists $h \colon T_E(R;f) \to H_E^*$ with $(E,h) \in \bA^{\mathrm{central}}_{T_E(R;f)}$ (\Cref{prop:central_factor}). 
\begin{lem}\label{lem:technical}
  Let $R$ be a Noetherian unstable algebra and suppose  $(E,f) \in \bA_R$. Suppose furthermore that $(E,h) \subseteq (V,\widetilde j)$ for $(V,\widetilde j) \in \bA_{T_E(R;f)}$. Then, $\widetilde{j}^{\#} \colon R \to H_E^* \otimes H_V^* \cong H_{E \oplus V}^*$ is equivalent to the composite 
  \[
R \xrightarrow{\rho_{R,(E,f)}} T_E(R;f) \xr{\widetilde j} H_V^* \xr{\mu^\ast} H_{E \oplus V}^*
  \]
  where the last map is induced by $\mu \colon E \oplus V \to V$ sending $(e,v) \mapsto \iota(e) + v$, where $\iota \colon E \to V$ denotes the inclusion.
\end{lem}
\begin{proof}
Because $(E,h) \in \bA^{\mathrm{central}}_{T_E(R;f)}$ we can apply \Cref{lem:sumuniqueness} to define $(\widetilde j \oplus h) \colon T_E(R;f) \to H^*_{E \oplus V}$. Moreover, the explicit construction given in \Cref{rem:dw_sum_construction}, along with the uniqueness part of \Cref{lem:sumuniqueness}, show that the following diagram commutes: 
\[
\begin{tikzcd}
R \arrow[d, "{\rho_{R,(E,f)}}"'] \arrow[rd, "\widetilde j^{\#}"]       &                  \\
T_E(R;F) \arrow[d, "\widetilde j"'] \arrow[r, "\widetilde j \oplus h"] & H^*_{E \oplus V} \\
H_V^* \arrow[ru, "\mu^*"']                                             &                 
\end{tikzcd}
\]
Therefore $\widetilde{j}^{\#} = (\widetilde j \oplus h) \circ \rho_{E,(E,f)} = \mu^* \circ \widetilde j \circ \rho_{R,(E,f)}$, as claimed.  
\end{proof}
This technical lemma is used in the following, which is a $T$-functor version of the observation that if $G$ is a group and $E$ and $V$  are elementary abelian $p$-subgroups of $G$, with $\mathcal{Z}(C_G(E)) < V < C_G(E)$, then $C_{C_G(E)}(V) \cong C_G(V)$, where $\mathcal{Z}(-)$ denotes the maximal central elementary abelian $p$-subgroup of a group.
\begin{prop}\label{lem:center_center}
  With assumptions as in the previous lemma, we have
  \[
T_V(T_E(R;f);\tilde j) \cong T_V(R; j)
  \]
  where $ j = \rho_{R,(E,f)}\circ \tilde j$.
\end{prop}
\begin{rem}The situation of the proposition is displayed in the following diagram:
  \[
\xymatrix{
  H_E^*& \ar[l]_-{\iota^*}   H_V^* & \\
  & T_E(R;f) \ar[ul]^{h} \ar[u]_{\tilde j} & \ar[l]^-{\rho_{R,(E,f)}} R  \ar[ul]_j
}
\]
\end{rem}
\begin{proof}
Applying \cite[Proposition 3.3]{DwyerWilkerson1992cohomology} and \Cref{lem:technical} there are isomorphisms
 \[
T_V(T_E(R;f);\tilde j) \cong T_{E \oplus V}(R;\tilde j^{\#}) \cong T_{E \oplus V}(R; \mu^{\ast} \circ j).
 \]
 Since $\mu$ is an epimorphism we have
 \[
T_{E \oplus V}(R;\mu^{\ast}\circ j) \cong T_V(R;j)
 \]
 by \Cref{lem:hls_lemma}. Combining these isomorphisms gives the result.
\end{proof}
\subsection{Central elements and the nilpotence degree}
The goal of this subsection is to improve the result \Cref{prop:injection} in the case $M = R$; more specifically, to prove that we only need to consider those $(E,f) \in \bA_R$ for which $(C,g) \subseteq (E,f)$. The proof will be based on the corresponding result for finite groups, due to Kuhn \cite[Theorem 4.4]{Kuhn2007Primitives}.

To begin, we recall that given a category $\cC$ the twisted arrow category $\cC^{\#}$ is the category whose objects are the morphisms of $\cC$, and a morphism from $f \colon C \to D$ to $f' \colon C' \to D'$ is a pair of morphisms $u \colon  C \rightarrow C'$, $v \colon D' \rightarrow D$ in $\cC$ such that the following diagram commutes:
  \[  \xymatrix { C \ar [r]^-{f} \ar [d]_{u} &  D \\ C' \ar [r]_-{f'} &  D'. \ar [u]_{v} } \]
The work of Henn--Lannes--Schwartz \cite{HennLannesSchwartz1995Localizations} discussed briefly in \Cref{sec:nilpotent_filtratio}, can be rephrased in terms of the twisted arrow category of $\bA_R$. In particular, following the discussion in \cite[(1.17.4)]{HennLannesSchwartz1995Localizations} the fundamental result \cite[Theorem 4.9]{HennLannesSchwartz1995Localizations} can equivalently be given as the statement that for $R$ a Noetherian unstable algebra and $M \in R-\cU$ there is a morphism
\begin{equation}\label{eq:limits_tac}
\begin{split}
M \to &\lim_{\alpha : (E,f) \to (E',f')}\left[ \Eq\left\{ H_{E}^* \otimes (T_{E'}(M;f'))^{\le n} \xrightrightarrows[\nu(\alpha)]{\mu(\alpha)} H_{E}^* \otimes (H_E^* \otimes (T_{E'}(M;f')))^{\le n} \right\} \right]
\end{split}
\end{equation}
which is localization away from $\Nil_n$. Here the limit is taken over the twisted arrow category $\bA_R^{\#}$, and $\Eq$ denotes the equalizer. The maps $\mu(\alpha)$ and $\nu(\alpha)$ are defined in \cite[(1.16.2)]{HennLannesSchwartz1995Localizations}; it will not prove important in what follows to have an explicit description of them, so we omit it.

Now let $R$ be a connected Noetherian unstable algebra with center $(C,g)$ and recall that Dwyer--Wilkerson have constructed a functor $\sigma \colon \bA_R \to (C,g) \downarrow {\bA_R}$, see \Cref{prop:dwfunctor}. Generalizing a result of Kuhn about finite groups \cite[Theorem 4.4]{Kuhn2007Primitives}, we now show the following.
\begin{thm}
  In the case $M = R$ of \eqref{eq:limits_tac}, the limit can be taken over the category $((C,g) \downarrow \bA_R)^{\#}$.
\end{thm}
\begin{proof}
  The proof will be the essentially the same as Kuhn's, just translated into the language of unstable algebras. To that end, let $\alpha \colon (E,f) \to (E',f')$ be a morphism in $\bA_R$, and let $\alpha_{C}$ denote the composite morphism in $\bA_R$
  \[
\alpha_C \colon (E,f) \xr{\alpha} (E',f') \to (C \circ E',\sigma(f',g)).
  \]
Let us now define morphisms $f_{\alpha} \colon \alpha_C \to \alpha$ and $g_{\alpha} \colon \alpha_C \to \sigma(\alpha)$ in $\bA_R^{\#}$ via the following commutative diagram in $\bA_R$:
\[
\begin{tikzcd}
{(E,f)} \arrow[r, Rightarrow,no head] \arrow[d, "\alpha"'] & {(E,f)} \arrow[r] \arrow[d, "\alpha_C"']                  & {(C \circ E,\sigma(f,g))} \arrow[d, "\sigma(\alpha)"'] \\
{(E',f')} \arrow[r]                                           & {(C \circ E',\sigma(f',g)))} \arrow[r, Rightarrow,no head] & {(C \circ E',\sigma(f',g))}
\end{tikzcd}
\]
Now \cite[Lemma 4.5]{Kuhn2007Primitives} goes through with an essentially unchanged proof: for any contravariant functor $F\colon \bA_R \to \Mod_{\F_p}$ such that for all $\alpha \colon (E,f) \to (E',f')$, $F(f_{\alpha}) \colon F(\alpha) \to F(\alpha_C)$ is an isomorphism, the canonical map
\[
\Xi \colon \lim_{\alpha \in \bA_R^{\#}} F(\alpha) \to \lim_{\alpha \in ((C,g) \downarrow \bA_R)^{\#}}F(\alpha)
\]
is an isomorphism. This applies in particular to $F(\alpha) = H_{E}^*$ and $F(\alpha) = T_{E'}(R;f')$; the first is clear, and the latter follows from \Cref{prop:dwfunctor}.  As with \cite[Theorem 4.4]{Kuhn2007Primitives} this completes the proof, as the limit in \eqref{eq:limits_tac} is built from these two examples by constructions that preserve isomorphisms.
\end{proof}
Using \cite[(1.17.4)]{HennLannesSchwartz1995Localizations} again, we can improve on their Theorem 4.9.
\begin{thm}
  Let $R$ be a connected Noetherian unstable algebra with center $(C,g)$, then the morphism
  \[
R \to \Eq\left\{ \prod_{(E,f)} H_E^* \otimes (T_E(R;f))^{\le n } \xrightrightarrows[\mu]{\nu}  \prod_{\alpha : (E',f') \to (E'',f'')} H_{E'}^* \otimes (H_{E'}^* \otimes T_{E''}(R;f'')))^{\le n}\right\}
  \]
  induced by the maps $\eta_{R,(E,f)}$ is localization away from $\Nil_n$ (the products in this formula are taken over all objects of $(C,g) \downarrow \bA_R$ resp.~over all morphisms of $(C,g) \downarrow \bA_R$).
\end{thm}
Finally, we have the following important corollary, which is the promised improvement of \Cref{prop:injection}.
\begin{cor}\label{prop:injection_improved}
  Let $R$ be a connected Noetherian unstable algebra with center $(C,g)$, then for $n \ge d_0(R)$ there is a monomorphism in $R_{fg}-\cU$
  \[
  \begin{tikzcd}
\phi'_R \colon R \arrow{r} & \displaystyle\prod_{(C,g) \subseteq (E,f) \in \bA_R} H_E^* \otimes T_E(R;f)^{\le n}.
\end{tikzcd}
  \]
  induced by the product of the maps $\eta_{R,(E,f)}$.
\end{cor}
  \section{The topological nilpotence degree of the central essential ideal}
  In this section we introduce the central essential ideal $\CEss(R)$ of a connected Noetherian algebra $R$, following the definition of Kuhn for compact Lie groups. We give an upper bound for $d_0(\CEss(R))$, and prove that $\CEss(R)$ is non-zero if and only if the depth of $R$ is equal to the rank of the center of $R$ (\Cref{defn:center}).
\subsection{The central essential ideal}\label{sec:cess}
We recall that in \cite{Kuhn2013Nilpotence} Kuhn defines the central essential ideal for a compact Lie group $G$ to be the kernel of the map
\[
\xymatrix{
H_G^* \ar[r] & \displaystyle \prod_{C(G) \lneq E} H^*_{C_G(E)},
}
\]
where the product is taken over those elementary abelian $p$-subgroups of $G$ strictly containing the maximal central subgroup $C(G)$. The analog for a general unstable algebra $R$ replaces $H_G^*$ with $R$ and $H_{C_G(E)}^*$ with components of the $T$-functor.
\begin{defn}
	Let $R$ be a connected Noetherian unstable algebra with center $(C,g) \in \bA_R$, then the central essential ideal $\CEss(R)$ is defined by
	\[
\xymatrix{
0 \ar[r] & \CEss(R) \ar[r] & R \ar[rr]^-{\prod \rho_{R,(E,f)}} && \displaystyle\prod_{(C,g) \subsetneq (E,f) \in \bA_R} T_E(R,f).
}
	\]
 Note that $\CEss(R)$ is independent of the choice of representative for the center. Moreover, by replacing $\bA_R$ by a choice of skeleton if necessary, we can assume this product is finite (see \Cref{prop:rector_props}).
\end{defn}
\begin{lem}\label{lem:cesscomodule}
	$\CEss(R)$ is a sub-$H_C^*$-comodule of $R$.
\end{lem}
\begin{proof}
	This is a consequence of \Cref{lem:comodulemorphism}.
\end{proof}
The main result of this section is the following. We refer the reader to \Cref{sec:appendix} for a brief discussion on the basic commutative algebra needed in this section, in particular, for the definition of the depth and dimension of an $R$-module.
\begin{thm}\label{thm:krulldimension}
	Let $R$ be a connected Noetherian unstable algebra with center $(C,g) \in \bA_R$. Let $c(R)$ be the rank of $C$, then the Krull dimension of the $R$-module $\CEss(R)$ is at most $c(R)$.
\end{thm}
The proof will require some preliminary results. We recall the following definitions, due to Henn \cite{Henn1996Commutative} and Powell \cite{powell}.

\begin{defn}
  Let $R$ be a Noetherian unstable algebra, and $M \in R-\cal{U}$.
  \begin{enumerate}
   	\item (Henn) The $T$-support of $M$ is
  \[
T-\supp(M) = \{ (E,f) \in \bA_R \mid T_E(M;f) \ne 0 \}.
 \]

 \item (Powell) The $R-\cal{U}$ transcendence degree of $M$ is
 \[
\TrD_{R-\cU}(M) = \sup\{\rank( E) \mid (E,f) \in T-\supp(M)\}.
 \]
   \end{enumerate}
\end{defn}
The following result justifies the terminology of the $R-\cU$ transcendence degree, see \cite[Proposition 7.2.6]{powell}.
\begin{prop}[Powell]\label{prop:powell_dim}
	Let $M \in R_{fg}-\cU$, then
	\[
\TrD_{R-\cU}(M) = \dim_R(M).
	\]
\end{prop}
The proof relies on the existence of Brown--Gitler modules $J_R(n)$ in the category $R-\cU$ (see \cite[Section 1.5]{Henn1996Commutative}), which represent the functor $M \mapsto (M^n)^*$, where $()^*$ is the vector space dual. Given $(E,f) \in \bA_R$, we define an injective object $I_{(E,f)}(n)$ in $R-\cU$ as $H_E^* \otimes J_{T_E(R;f)}(n)$ \cite[Proposition 1.6]{Henn1996Commutative}. In fact, if $R$ is a Noetherian unstable algebra, then $I_{(E,f)}(n)$ is even injective in $R_{fg}-\cU$. From the definitions (see also \cite[Lemma 6.1.7]{powell}) we have
\begin{equation}\label{eq:tfunctorinjec}
\Hom_{R-\cU}(M,I_{(E,f)}(n)) \cong (T_E(M;f)^n)^*.
\end{equation}
We now present the proof of \Cref{prop:powell_dim}.
\begin{proof}(Powell)
Since Powell's work is not published, we sketch Powell's proof here. To that end, let $\overline R = R/\Ann_R(M)$, which is a Noetherian unstable algebra (note that the annihilator ideal is closed under the action of the Steenrod algebra) such that $\alpha \colon R \to \overline R$ is a morphism of unstable algebras, and let $\overline M \in \overline{R}_{fg}-\cal{U}$ denote the object obtained by inducing $M$ along the morphism $\alpha$, so that $M \cong \alpha^* \overline{M}$. Standard base change results about Lannes' $T$-functor show that $\TrD_{R-\cU}(M) = \TrD_{\overline{R}-\cU}(\overline M)$ see \cite[Proposition 7.2.2(1)]{powell} (if the reader prefers a published reference, this is also easily deduced from the formulas on page 1756 of \cite{NotbohmRay2010DavisJanuszkiewicz}).

  Now $\dim_R(M) = \dim(\overline R) = \max\{ \rank(E) \mid (E,f) \in \bA_{\overline R}\} \ge \TrD_{\overline R-\cal{U}}(\overline M)$ by \Cref{prop:rector_props}, which gives one inequality.

  For the reverse inequality, we recall that the Dickson invariants are defined by
  \[
D_n = (H^*_{(\Z/2)^n})^{\GL_n(\Z/2)} \text{ for } p = 2
  \]
  and
  \[
D_n = (P_n)^{\GL_n(\Z/p)} \text{ for } p > 2
  \]
  where $P_n$ is the subalgebra of $H^*_{(\Z/p)^n}$ generated by $\beta H^1_{(\Z/p)^n}$. As is well known, $D_n \cong \F_p[c_1,\ldots,c_n]$. Then, for $s$ another non-negative integer, one lets $D_{n,s}$ denote the subalgebra of $D_n$ whose elements are the $p^s$-th powers of elements $D_n$, which naturally obtains an action of the Steenrod algebra. Specifically, $D_{n,s} \cong \F_p[c_1^{p^s},\ldots,c_n^{p^s}]$. Suppose now that $\dim(\overline R) = n$, then by \cite[Appendix A]{BourguibaZarati1997Depth} there exists a natural number $s$ and a monomorphism of unstable algebras $\iota \colon D_{n,s} \to \overline R$ for which $\overline{R}$ is a finitely-generated $D_{n,s}$-module. We let $\omega_{\iota}$ denote the image of the top Dickson invariant $c_n^{p^s}$. Because $M$ is Noetherian, the localization $\overline M[\omega_{\iota}^{-1}]$ is non-trivial.

  By \cite[Theorem 1.9]{Henn1996Commutative} there exists an embedding in $\overline R - \cU$
  \[
\overline M \hookrightarrow \bigoplus_{i \in \cal{I}} I_{(E_i,f_i)}(a_i)
  \]
  where each component is non-trivial, and $(E_i,f_i) \in \bA_{\overline R}$, so that, in particular by \Cref{prop:rector_props}(2), $\rank(E_i) \le n$. Using exactness of localizations, there exists an $i \in \cal{I}$ for which $I_{(E_i,f_i)}(a_i)[\omega_{\iota}^{-1}] \ne 0$. By \cite[Lemma 7.1.4]{powell} we have $\rank(E_i) = n$. By \eqref{eq:tfunctorinjec} $T_{E_i}(\overline M;f_i) \ne 0$, and hence $T-\supp(\overline M) \ge n = \dim_R(M)$. In particular, $\TrD_{\overline R-\cal{U}}(\overline M) \ge \dim_R(M)$, as required.
\end{proof}
We will need the following computation, which is an almost immediate consequence of \cite[Lemma 3.6]{Henn1996Commutative}. The proof is given in \cite[Proposition 3.14]{heard_depth}.
\begin{prop}\label{prop:supportcalc}
Let $R$ be a Noetherian unstable algebra and $M \in R_{fg}-\cal{U}$, then $$\TrD_{R-\cal{U}}(H_E^* \otimes T_E(M;f)^{\le n}) \le \rank(E).$$
 \end{prop}
 Finally, we also need the following result, also due to Powell \cite[Proposition 7.3.1]{powell}. The proof is also given in \cite[Proposition 3.17]{heard_depth}.
 \begin{prop}[Powell]\label{prop:powell_mono}
   Let $M$ be non-trivial and $ M \hookrightarrow N$ a monomorphism in $R_{fg}-\cU$, then
   \[
\TrD_{R-\cU}(M) \ge \depth_R(N).
   \]
 \end{prop}
 With these preparations, we can now prove \Cref{thm:krulldimension}.
 \begin{proof}[Proof of \Cref{thm:krulldimension}]
If $\CEss(R) = 0$ then the result is clear, thus we can assume that $\CEss(R) \ne 0$. By \Cref{prop:injection_improved} we can find $n$ large enough so that
	\[
\xymatrix{\lambda \colon R \ar[r]& \displaystyle \prod_{(C,g) \subseteq (E,f)}H_E^* \otimes T_E(R;f)^{\le n}}
	\]
	is a monomorphism in $R_{fg}-\cU$.
	 We factor $\lambda$ as a product $\lambda = \lambda_{>c} \times \lambda'$ where
\[
\xymatrix{\lambda_{> c} \colon R \ar[r]& \displaystyle \prod_{(C,g) \subsetneq (E,f)}H_E^* \otimes T_E(R;f)^{\le n}}
\]
and
\[
\xymatrix{\lambda' \colon R \ar[r]& H_C^* \otimes T_C(R;g)^{\le n}}
\]
Recall that $\eta_{R,(E,f)} \cong \kappa_{R,(E,f)} \circ \rho_{R,(E,f)}$ so that we can factor $\lambda_{>c}$:
\[
\begin{tikzcd}
R \arrow[rr, "\lambda_{>c}"] \arrow[rd, "\rho_{>c}"'] &                                                      & {\prod_{(C,g) \subsetneq (E,f)} H_E^* \otimes T_E(R;f)^{\le n}} \\
                                                      & {\prod_{(C,g) \subsetneq (E,f)} T_E(R;f)} \arrow[ru] &
\end{tikzcd}
\]
where $\rho_{>c}$ is the product of the maps $\rho_{R;(E,f)}$ for $(C,g) \subsetneq (E,f)$. In particular, $\CEss(R) = \ker(\rho_{>c})$. The factorization shows that $\CEss(R)$ is contained in the kernel of $\lambda_{> c}$, and since $\lambda$ is injective, we deduce that the restriction of $\lambda'$ to $\CEss(R) \subset R$ is injective. We deduce that $\TrD_{R-\cU}(\CEss(R)) \le \TrD_{R-\cU}(H_C^* \otimes T_C(R;f)^{<n}) \le c(R)$, where the last inequality uses \Cref{prop:supportcalc}. By \Cref{prop:powell_dim} we have $\TrD_{R}(\CEss(R)) =\dim_R(\CEss(R)) \le c(R)$, as claimed.
  \end{proof}
For the following, we recall that if $R$ is a connected Noetherian unstable algebra with center $(C,g)$ then the depth of $R$ is always at least equal to the rank of $C$, see \Cref{cor:duflot}.
  \begin{cor}\label{cor:carlson}
  	Let $R$ be a connected Noetherian algebra with center $(C,g)$ and let $c(R) = \rank(C)$. If $\CEss(R) \ne 0$, then $\depth(R) = c(R)$, and $\dim_R(\CEss(R)) = c(R)$.
  \end{cor}
  \begin{proof}
  	Assume that $\CEss(R) \ne 0$, then by \Cref{cor:duflot}, the previous result, and \Cref{prop:powell_mono} we have
  	\[
c(R) \le \depth(R) \le \TrD_{R-\cU}(\CEss(R)) \le c(R).
  	\]
  	Thus, $\depth(R) = \TrD_{R-\cU}(\CEss(R)) = \dim_R(\CEss(R)) = c(R)$.
  \end{proof}
We will prove the converse in \Cref{thm:cessnonzero}.
  \subsection{Primitives and indecomposables}\label{sec:prim_indec}
Let $R$ be a connected Noetherian unstable algebra with center $(C,g)$.\footnote{We allow the case where the center is trivial. In this case a $H_C^*$-comodule is simply an $\F_p$-module.} Since $(C,g) \in \bA_R$ the morphism $g$ is finite, and so the following is well defined.
\begin{defn}
	Let $e(R)$ denote the maximum degree of a generator (with respect to a minimal generating set) for $H_C^\ast$ as a $R$-module, or equivalently the top nonzero degree of the finite dimensional Hopf algebra $H_C^* \otimes_R \F_p$.
	\end{defn}
We recall from \Cref{prop:restriction_comodule,cor:Borel} that $g \colon R \to H_C^*$ is a morphism of $H_C^*$-comodules, and that there is a basis $x_1,\ldots,x_c$ for $H_C^1$ such that
\begin{equation}
\im(g) = \begin{cases}
  \F_2[x_1^{2^{j_1}},\ldots,x_c^{2^{j_c}}] &\text{ if } p=2\\
  \F_p[y_1^{p^{j_1}},\ldots,y_b^{p^{j_b}},y_{b+1},\ldots,y_c] \otimes \Lambda(x_{b+1},\ldots,x_c) &
  \text{ if } p \text{ is odd,}
\end{cases}
\end{equation}
for some natural numbers $j_1 \ge j_2 \ge \cdots $, and where $y_i = \beta x_i$ for $\beta$ the Bockstein homomorphism. We then have
  \[
e(R) = \sum_{i=1}^c (a_i-1).
  \]
  where
  \[
a_i = \begin{cases}
  2^{j_i} & p = 2\\
  2p^{j_i} & p \text{ odd, and } 1 \le i \le b\\
  1 & \text{otherwise.}
\end{cases}
  \]

	In order to proceed, we need one more definition, due to Kuhn \cite[Definition 2.15]{Kuhn2013Nilpotence}.
	\begin{defn}
		A Duflot algebra of $R$ is a subalgebra $B \subseteq R$ that maps isomorphically to $K = \im(R \to H_{C}^*)$.
	\end{defn}
Since the image $K$ is a free graded-commutative algebra over $\F_p$, such Duflot algebras always exist (as the natural epimorphism $R \to K$ always splits).

Given a Noetherian unstable algebra $R$, we fix a Duflot algebra $B \subseteq R$.
\begin{defn}
	If $M$  is a graded $B$-module, then the space of indecomposables is
	\[
Q_BM \myeq M \otimes_B \F_p = M/B^{>0}M.
	\]
	We let $e_{\text{indec}}(M)$ be its largest nonzero degree, or $-\infty$ if $M = 0$.
\end{defn}
As shown in \Cref{lem:cesscomodule}, $\CEss(R)$ is a sub $H_C^*$-comodule of $R$. Moreover, it is an unstable module, as it is the kernel of a morphism of unstable modules, and the comodule structure map is a morphism of unstable modules. Comodules with this additional structure are called unstable $H_C^*$-comodules in \cite{Kuhn2013Nilpotence}.
\begin{defn}
 	Let $M$ be an unstable $H_C^*$-comodule, then the modules of primitives is
 	\[
P_CM = \{ x \in M \colon \Psi_M= 1 \otimes x\} = \Eq\{ M \xrightrightarrows[i]{\Psi_M} H_C^* \otimes M\},
 	\]
 	where $\Psi_M \colon M \to H_C^* \otimes M$ is the comodule structure map and $i(x) = 1 \otimes x$

 	We let $e_{\prim}(\CEss(R))$ denote the supremum of the degrees in which $P_C(\CEss(R))$ is non-zero, with the convention that this is $-\infty$ if $\CEss(R) = 0$.
 \end{defn}
 \begin{rem}
 	If $R$ has trivial center, then the constructions still make sense, where $e(R) = 0$, and $Q_B\CEss(R) \cong P_C\CEss(R) \cong \CEss(R)$.
 \end{rem}
\begin{rem}
   Note that $Q_BM$ will not necessarily be an unstable module because $B$ is not necessarily closed under Steenrod operations. On the other hand $P_CM$ is always an unstable module.
 \end{rem}

The following lemma is proved by Totaro for Chow rings of finite $p$-groups \cite[Lemma 12.10]{Totaro2014Group} - the proof goes through essentially without change here.
\begin{lem}\label{lem:freeinjective}
  Let $R$ be a Noetherian unstable algebra, with center $(C,g)$. Let $M$ be a non-negatively graded $R$-module that is also a $H_{C}^*$-comodule, such that the morphism $R \otimes M \to M$ is a morphism of $H_{C}^*$-comodules. Let $B \subseteq R$ be a Duflot algebra. Then,
  \begin{enumerate}
      \item $M$ is a free $B$-module
      \item The composite $P_CM \hookrightarrow M \twoheadrightarrow Q_BM$ is injective.
  \end{enumerate}
\end{lem}
\begin{proof}
  We follow Totaro \cite[Lemma 12.10]{Totaro2014Group}. To that end, let $L = \ker(g \colon R \to H_C^*)$, and let $M_i = L^iM \subset M$ for $i \ge 0$, where $L^i$ denotes the $\F_p$-linear span of all products of $i$ elements of the ideal $L$. This gives a filtration of $M$ by $R$-modules.

  By \Cref{prop:restriction_comodule} $g \colon R \to H_C^*$ is a morphism of $H_C^*$-comodules, and hence $L$ is a sub-$H_C^*$-comodule of $R$. Let $\Psi_M \colon M \to H_C^* \otimes M$ denote the $H_C^*$-comodule structure map for $M$, and $\Psi_{R}$ the corresponding comodule structure map for $R$. By assumption there is a commutative diagram
  \[
\begin{tikzcd}
  R \otimes M \arrow{r} \arrow[swap]{d}{\Psi_{R} \otimes \Psi_M} & M \arrow{d}{\Psi_M}\\
  (H_{C}^* \otimes R) \otimes  (H_C^* \otimes M) \arrow{r} & H_C^* \otimes M
\end{tikzcd}
  \]
  Since $L$ is a sub-$H_C^*$-comodule, this diagram implies that $LM \subset M$ is a $H_C^*$-comodule. By induction we see that $M_i$ is a sub-$H_C^*$-comodule for all $i \ge 0$.

  For each $i \ge 0$ it follows that $\gr_iM = M_i/M_{i+1}$ is a $H_C^*$-comodule. It is also a module over $K = \im(R \xr{g} H_C^*)$, which we have seen is a sub-Hopf algebra of $H_C^*$. By assumption, the $H_C^*$-comodule structure and the $R$-module structure are compatible. Applying a lemma of Kuhn \cite[Lemma 5.2]{Kuhn2007Primitives} we deduce that $\gr_iM$ is a free $B$-module, and the composite $P_C(\gr_iM) \hookrightarrow \gr_iM \twoheadrightarrow Q_B(\gr_iM)$ is injective, for each $i \ge 0$. The filtration of $M$ given by the $M_i$ is separated, and so the fact that each  $\gr_iM$ is a free $B$-module implies that $M$ is also a free $B$-module.
\end{proof}
\begin{thm}\label{thm:pcompactcess} Let $R$ be a connected Noetherian unstable algebra with center $(C,g)$, and fix a Duflot algebra $B$ of $R$.
\begin{enumerate}
	\item $\CEss(R)$ is a finitely-generated free $B$-module, i.e., a Cohen--Macaulay module.
	\item The composite $P_C \CEss(R) \hookrightarrow \CEss(R) \twoheadrightarrow Q_B \CEss(R)$ is monic.
  \item There is an exact sequence
\[
\begin{tikzcd}
0 \arrow{r} & Q_B\CEss(R) \arrow{r}&  Q_BR \arrow{r} & \displaystyle \prod_{(C,g) \subsetneq (E,f)} Q_BT_E(R,f)
\end{tikzcd}
\]
\end{enumerate}
\begin{proof}
	Everything in (1) and (2) except for the claim that $\CEss(R)$ is finitely-generated is a consequence of the previous lemma with $M = \CEss(R)$. Now, $B$ has Krull dimension equal to the rank of $C$, namely $c(R)$. Since we know $\CEss(R)$ is a free $B$-module, it suffices to check that the Krull dimension of $\CEss(R)$ is at most $c(R)$, which is \Cref{thm:krulldimension}.

  For (3), consider the left exact sequence
	\[
\begin{tikzcd}[column sep=0.5  in]
0 \arrow{r} & \CEss(R) \arrow{r} & R \arrow{r}{\prod \rho_{R,(E,f)}} & \displaystyle\prod_{(C,g) \subsetneq (E,f)} T_E(R,f)
\end{tikzcd}
	\]
Note that these are $R$-modules and $H_C^*$-comodules in a compatible way. We can apply \Cref{lem:freeinjective} to the images and cokernels of these maps to deduce that they are free $B$-modules. It follows that the maps split, and we have an exact sequence as claimed.
\end{proof}
\end{thm}
\begin{cor}\label{cor:d0cess}
	If $\CEss(R) \ne 0$, then we have $d_0(\CEss(R)) = e_{\prim}(\CEss(R))$. In general, $e_{\prim}(\CEss(R)) \le e_{\indec}(\CEss(R)) < \infty$.
\end{cor}
\begin{proof}
	The stated inequality is an immediate consequence of \Cref{thm:pcompactcess}(2) applied to $M = \CEss(R)$.

  For the first claim we use \cite[Lemma 2.11]{Kuhn2013Nilpotence} in which Kuhn proves that any unstable $H_C^*$-comodule with $P_CM$ finite-dimensional has the property that $d_0M = e_{\prim}(M)$ (under our conventions this is only true if $M \ne 0$). Using \Cref{thm:pcompactcess}(2) again, we see that if $Q_A\CEss(R)$ is finite-dimensional, then so is $P_C\CEss(R)$. But it is clear that if $\CEss(R)$ is a finitely-generated $B$-module, then $Q_B \CEss(R)$ is finite-dimensional, and this is a consequence of \Cref{thm:pcompactcess}(1). The first part of the corollary then follows by applying Kuhn's lemma to $M = \CEss(R)$.
\end{proof}
\subsection{Regularity and \texorpdfstring{$e_{\indec}(\CEss(R))$}{eindecCEssR}}\label{sec:regindec}

We now give the following version of \cite[Proposition 2.27]{Kuhn2013Nilpotence}. This proposition is the first point of the paper we need to make some assumptions on the Duflot algebra. In particular, we make the following hypothesis, which is in effect for the rest of this section.
\begin{hyp}\label{hyp:duflot}
$R$ is a connected Noetherian unstable algebra whose Duflot algebra $B$ is polynomial.
\end{hyp}
\begin{rem}
This always holds if $p = 2$, or if $R$ is concentrated in even degrees (i.e., $R$ is naturally an element of $\cal{U}'$, see \Cref{rem:odd_prime}).
\end{rem}
For the following, we let $\frak m = R^{> 0}$ denote the maximal homogeneous ideal of $R$, and let
\[
H_{\frak m}^0(M) = \{ x \in M \mid \text{ there exists } n \in \mathbb{N} \text{ with } m^{\frak n}x = 0\}.
\]
denote the $\frak m$-torsion functor for an $R$-module $M$ (see \Cref{sec:appendix}). This has right derived functors, $H_{\frak m}^i(M)$, the $i$-th local cohomology of $M$. We note that since $R$ and $M$ are graded, so are these local cohomology modules, although we will usually suppress the internal grading. 
\begin{prop}\label{prop:esslocalcohom}
Let $R$ be a connected Noetherian unstable algebra with Duflot algebra $B$, then
  \[
  Q_B\CEss(R) = H_{\frak m}^0(Q_B\CEss(R)) = H_{\frak m}^0(Q_BR).
  \]
\end{prop}
\begin{proof}
  The first equality follows because $Q_B\CEss(R)$ is finite-dimensional. For the second, consider the left exact sequence of \Cref{thm:pcompactcess}(3):
\[
\begin{tikzcd}
0 \arrow{r} & Q_B\CEss(R) \arrow{r}&  Q_BR \arrow{r} & \displaystyle \prod_{(C,g) \subsetneq (E,f)} Q_BT_E(R,f)
\end{tikzcd}
\]
This gives a left exact sequence
\[
\begin{tikzcd}
0 \arrow{r} & H^0_{\frak m}Q_B\CEss(R) \arrow{r}&  H^0_{\frak m}Q_BR \arrow{r} & \displaystyle \prod_{(C,g) \subsetneq (E,f)} H^0_{\frak m}Q_BT_E(R,f)
\end{tikzcd}
\]
Thus we must show that whenever $(C,g) \subsetneq (E,f)$, we have $ H^0_{\frak m}Q_BT_E(R,f) = 0$. Fix such a pair $(E,f)$.

By assumption, the Duflot algebra $B$ is polynomial, say $B \cong \F_p[f_1,\ldots,f_c]$, and moreover by \Cref{lem:freeinjective} $H_E^*$ is a free $B$-module, so that the sequence $f_1,\ldots,f_c$ is regular by \Cref{lem:depth_regular}. The cohomology $H_E^*$ is Cohen--Macaulay, and hence $Q_BH_E^*$ is also Cohen--Macaulay, of Krull dimension $r - c$, where $r$ is the $p$-rank of $E$ (as the quotient of a Cohen--Macaulay ring by a regular sequence is still Cohen--Macaulay, see \cite[Theorem 2.1.3]{BrunsHerzog1993CohenMacaulay}). We note that $r>c$ by the assumption that $(C,g) \subsetneq (E,f)$. It follows that $\depth(Q_BH_E^*) = r-c > 0$.

Because $H_E^*$ is a finitely generated $R$-module via $f$ so is the quotient ring $Q_BH_E^*$. It follows from \Cref{lem:fgdepth} that
\[
r - c = \depth(Q_BH_E^*) =  \depth_R(Q_BH_E^*).
\]
In particular, by \Cref{lem:depth_regular} there exists elements $y_i \in \frak m = R^{>0}$ such that $Q_BH_E^*$ is a finitely generated free module over the graded polynomial subring $S \cong k[y_1,\ldots,y_{r-c}] \subseteq R$.

 Since $S$ has dimension $r - c>0$, we can find a non-zero element $\ell$ with positive degree which is a non-zero divisor on $Q_BH_E^*$. It follows that the sequence $f_1,\ldots,f_c,\ell \in R$ restricts to a regular sequence in $H_E^*$.  By \Cref{prop:central_factor} there exists a $h \colon T_E(R;f) \to H_E^*$ such that $(E,h)$ is central in $T_E(R,f)$, and so \Cref{thm:duflot_regular} applies to show that the sequence $f_1,\ldots,f_c,\ell$ is regular in $T_E(R;f)$. It follows that $\ell \in \frak m$ restricts to a non-zero divisor on $Q_BT_E(R;f)$, and so by \cite[Lemma 2.1.1(i)]{BrodmannSharp2013Local} $H_{\frak m}^0Q_BT_E(R;f) = 0$ , as required.
\end{proof}
\begin{rem}
  The proof actually shows that, even when \Cref{hyp:duflot} does not hold, we still have $Q_B\CEss(R) = H_{\frak m}^0(Q_B\CEss(R)) \subseteq H_{\frak m}^0(Q_BR)$.
\end{rem}

Given an $R$-module $M$, we let $a_i(R,M)$ be the maximum degree of a non-zero element of $H^i_{\frak m}(M)$ (with the convention that this is $\infty$ if unbounded, or $-\infty$ if $H_{\frak m}^i(M) = 0$). The Castelnuovo--Mumford regularity of $M$ is defined as
\[
\reg(R,M) = \sup_i \{a_i(R,M)+i\},
\]
 see \cite{Symonds2010CastelnuovoMumford} for the basic properties of regularity. If $R = M$, then we write $\reg(R)$ and $a_i(R)$. The main result of this section is the following. We remind the reader we work under \Cref{hyp:duflot}. 
\begin{thm}\label{thm:d0cess}
  Let $R$ be a connected Noetherian unstable algebra, then $e_{\indec}(\CEss(R)) \le e(R) + \reg(R)$, and hence if $\CEss(R) \ne 0$, we have
  \[
d_0(\CEss(R)) \le e(R) + \reg(R).
  \]
  More specifically, we have
  \[
  \begin{split}
e_{\indec}(\CEss(R)) &= e(R)  + a_{c(R)}(R) + c(R) \le e(R) + \reg(R).
\end{split}
  \]
\end{thm}
\begin{proof}
We first show that $e_{\indec}(\CEss(R)) = e(R)  + a_{c(R)}(R) + c(R)$. This is the claim that the top non-zero degree of $Q_B\CEss(R)$ is $e(R)  + a_{c(R)}(R) + c(R)$. By \Cref{prop:esslocalcohom} it is equivalent to show that the top non-zero degree of $H_{\frak m}^0(Q_BR)$ is $e(R)  + a_{c(R)}(R) + c(R)$. 

Now, by definition $H_{\frak m}^{c(R),e}(R) = 0$ for all $e > a_{c(R)}(R)$ and $H_{\frak m}^{c(R),a_{c(R)}(R)}(R) \ne 0$ (here we are explicitly writing the internal degree of the local cohomology module). If we choose algebra generators $z_1,\ldots,z_c$ for the Duflot algebra, then $|z_1| + \cdots |z_c| = c(R) + e(R)$, $Q_BR = R/(z_1,\ldots,z_c)$ and the long exact sequence in local cohomology shows that  $H_{\frak m}^{0,e(R)+c(R)+e}(Q_BR) = 0$ for all $e>a_{C(R)}(R)$ and
\[
H_{\frak m}^{0,e(R)+c(R)+a_{c(R)}(R)}(Q_BR) = H_{\frak m}^{c(R),a_{c(R)}(R)}(R) \ne 0,
\]
compare \cite[Corollary 2.25 and Proposition 2.26]{Kuhn2013Nilpotence}. In particular, $e_{\indec}(\CEss(R)) = e(R) + a_{c(R)}(R) + c(R)$, as claimed.  The subsidiary claim about $d_0(\CEss(R))$ is then a consequence of \Cref{cor:d0cess}.
\end{proof}

We now turn to a form of Carlson's depth conjecture \cite{Carlson1995Depth} or \cite[Question 12.5.7]{CarlsonTownsleyValeriElizondoZhang2003Cohomology}. We recall that, for a finite group $G$, Carlson has conjectured that if the product of restriction maps
\[
H_G^* \longrightarrow \prod_{\rank(E) = s}H^*_{C_G(E)}
\]
is injective, then $\depth(H_G^*) \ge s$. This is Carlson's depth conjecture. It is a theorem of Duflot \cite{Duflot1981Depth} that the depth of $H_G^*$ is always equal to at least the $p$-rank of the center of a Sylow $p$-subgroup of $G$. If equality holds, we will say that $H_G^*$ has minimal depth. Suppose that $G$ is a $p$-group and $H_G^*$ has minimal depth, then Carlson's depth conjecture is that $\CEss(H_G^*) \ne 0$, and has been proven in this case by Green \cite{Green2003Carlsons} and Kuhn \cite[Theorem 2.13]{Kuhn2007Primitives}. 

In \cite{heard_depth} the author generalized Duflot's result on depth for connected Noetherian unstable algebras; the depth is always at least $c(R)$, the rank of the center of $R$ (see \Cref{cor:duflot}), and we say that $R$ has minimal depth if $\depth(R) = c(R)$. One can then ask that if $R$ has minimal depth, then is $\CEss(R) \ne 0$? (note that we have already proved the converse of this in \Cref{cor:carlson}). This is part of the content of the following. 
\begin{thm}\label{thm:cessnonzero}
   The central essential ideal $\CEss(R)$ is non-zero if and only if the depth of $R$ is minimal, i.e.,  $\depth (R) = c(R)$. Moreover, in this case $\CEss(R)$ is a Cohen--Macaulay $R$-module of dimension $c(R)$.
 \end{thm}
 \begin{proof}
The only if direction is \Cref{cor:carlson}, so we prove the converse. To this end, suppose that $\depth(R) = c(R)$, so that $H_{\frak m}^{c(R)}(R) \ne 0$ by \Cref{prop:depth_local_cohom}. By \Cref{thm:d0cess} we have $e_{\indec}(\CEss(R)) \ge 0$, and hence $\CEss(R) \ne 0$.

For the second claim, observe that we have
\[
c(R) \le \depth_R(\CEss(R)) \le \dim_R(\CEss(R)) = c(R)\]
 by \Cref{cor:carlson,thm:pcompactcess}.
 \end{proof}
 \section{The topological nilpotence degree of a Noetherian unstable algebra}
 In this section we introduce the $p$-central defect of a Noetherian unstable algebra, which is the analog of $p$-centrality for finite groups. For unstable algebras of $p$-central defect 0, we have an immediate estimate for $d_0(R)$. In general, we use an inductive argument to prove the following, the main result of the paper. 
 \begin{thm}\label{thm:main_unstable_algebra}
  Let $R$ be a connected Noetherian unstable algebra with center $(C,g)$, and suppose that $T_E(R;f)$ satisfies the assumptions of \Cref{hyp:duflot} for all $(C,g) \subseteq (E,f)$, then
  \[
d_0(R) \le  \underset{\substack{(C,g) \subseteq (E,f) \in \bA_R \\ \depth(T_E(R;f)) = c(T_E(R;f))}}\max \{e(T_E(R;f)) + \reg(T_E(R;f)) \}.
  \]
\end{thm}
The proof will be given in \Cref{sec:top_degree_proof}.
 \subsection{The \texorpdfstring{$p$}{p}-central defect of a Noetherian unstable algebra}
 \begin{defn}
   Let $R$ be a connected Noetherian unstable algebra with center $(C,g)$. Let $c(R)$ be the rank of $C$, and
   \[
p(R) = \max \{ \rank(E) \mid (E,f) \in \bA_R \}.
   \]
   The $p$-central defect of $R$ is $p(R)- c(R)$.
\end{defn}
\begin{lem}\label{rem:central_defect}
We have
\[
c(R) \le \depth(R) \le \dim(R) = p(R)
\]
If particular, the $p$-central defect is always greater than or equal to zero, and if it is zero, then $\depth(R) = \dim(R) = c(R)$, so that $R$ is a Cohen--Macaulay ring. 
\end{lem}
\begin{proof}
  We have $c(R) \le \depth(R)$ by the author's generalization of Duflot's theorem (\Cref{cor:duflot}), the inclusion $\depth(R) \le \dim(R)$ always holds, and $\dim(R) = p(R)$ by  \Cref{prop:rector_props}(2). The result of the lemma is then clear. 
\end{proof}
\begin{rem}
  The Cohen--Macaulay defect of $R$ is defined as $\dim(R) =p(R) - \depth(R)$. By Duflot's depth theorem (\Cref{cor:duflot}) $\depth(R) \ge c(R)$, so that the Cohen--Macaulay defect is always less than or equal to the $p$-central defect of $R$.
\end{rem}
\begin{lem}\label{lem:central_defect}
  Let $R$ be a connected Noetherian unstable algebra with center $(C,g)$. If $(C,g) \subseteq (E,f)$, then the $p$-central defect of $T_E(R;f)$ is less than or equal to $R$, with equality if and only if $(C,g) \simeq (E,f)$. 
\end{lem}
\begin{proof}
  We first claim that $p(T_E(R;f)) \le p(R)$. Indeed, if $(V,\tilde g) \in \bA_{T_E(R;f)}$, then we can precompose with $\rho_{R,(E,f)} \colon R \to T_E(R;f)$ to get a pair $(V,g) \in \bA_R$ for $g = \tilde g \circ \rho_{R,(E,f)}$. On the other hand, we recall there exists $h \colon T_E(R;f) \to H_E^*$ such that $(E,h)$ is central in $\bA_{T_E(R;f)}$, see \Cref{prop:central_factor}. Thus, $c(R) \le c(T_E(R;f))$, with equality if and only if $(E,f)$ is central in $\bA_R$, in which case $(E,f) \simeq (C,g)$. Combining these two inequalities we see that
  \[
p(T_E(R;f)) - c(T_E(R;f)) \le p(R) - c(R),
  \]
  hence the result.
\end{proof}
Algebras of $p$-central defect 0 have several nice properties, as we now explain. We note that in this case the product in the definition of $\CEss(R)$ is taken oven the empty set, and hence this product trivial algebra, so that in this case $\CEss(R) \cong R$.

 The following characterization of unstable algebras of $p$-central defect 0 was suggested by the referee. 
\begin{lem}
  Let $R$ be a connected Noetherian unstable algebra with center $(C,g)$, then $R$ has $p$-central defect 0 if and only if there is an equivalence of categories $\bA_R \simeq \bA_{H_C^*}$. 
\end{lem}
\begin{proof}
If $R$ has $p$-central defect greater than 0, then there exists a pair $(E,f) \in \bA_R$ with $(C,g) \subsetneq (E,f)$, and such an $(E,f)$ cannot be equivalent to an object $\bA_{H_C^*}$ as $\rank(C) \lneq \rank(E)$. Now suppose that $R$ has $p$-central defect equal to zero, and let $(E,f) \in \bA_R$. Note that $(E,f) \subseteq (C,g)$ as $(C,g)$ is the unique maximal object of $\bA_R$. In particular, there is a morphism $\alpha \colon H_C^* \to H_E^*$, and clearly the pair $(E,\alpha)$ defines an object of $\bA_{H_C^*}$. This defines a functor $F \colon \bA_R \to \bA_{H_C^*}$.  Conversely, if $(V,j) \in \bA_{H_C^*}$, then $(V,j \circ g) \in \bA_R$, and it is easy to see that this is functorial, and provides an inverse to $F$. 
\end{proof}
\begin{ex}
  Let $G$ be a $p$-group, then we will see in \Cref{thm:group_props} that $\bA_{H_G^*}$ is equivalent to the category $\bA_G$ whose objects are the elementary abelian $p$-subgroups of $G$. Moreover, the center of $H_G^*$ is just the group-theoretic center of $G$. It follows that $G$ has $p$-central defect 0 if and only if the maximal central elementary abelian $p$-subgroup is maximal among all elementary abelian $p$-subgroups of $G$. Such groups are known as $p$-central groups. For example, $G = Q_8$ (the quaternion group of order 8) is 2-central, while $D_8$ (the dihedral group of order 8) is not 2-central; in our terminology, $H_{Q_8}^*$ has 2-central defect 0, and $H_{D_8}^*$ has 2-central defect 1. In terms of the previous lemma, the elementary abelian subgroup lattices of $Q_8$ and $D_8$ are given respectively (up to conjugacy) by
  
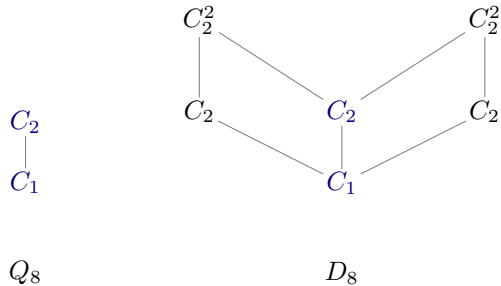
\[
\begin{tikzpicture}[scale=1.0,sgplattice]
  \node[char] at (2,0) (1) {\gn{C1}{C_1}};
  \node[char] at (2,0.803) (2) {\gn{C2}{C_2}};
  \draw[lin] (1)--(2) ;
  \node [below=1cm, align=flush center,text width=4cm] at (1)
        {
             $Q_8$
        };
\end{tikzpicture}
\begin{tikzpicture}[scale=1.0,sgplattice]
  \node[char] at (2,0) (1) {\gn{C1}{C_1}};
  \node[char] at (2,0.953) (2) {\gn{C2}{C_2}};
  \node at (0.125,0.953) (3) {\gn{C2}{C_2}};
  \node at (3.88,0.953) (4) {\gn{C2}{C_2}};
  \node at (0.125,2.17) (6) {\gn{C2^2}{C_2^2}};
  \node at (3.88,2.17) (7) {\gn{C2^2}{C_2^2}};
  \draw[lin] (1)--(2) (1)--(3) (1)--(4)  (2)--(6) (3)--(6) (2)--(7)
     (4)--(7)     ;
       \node [below=1cm, align=flush center,text width=4cm] at (1)
        {
             $D_8$
        };
\end{tikzpicture}
  \]
  In both cases, the center is isomorphic to $C_2$, however $\bA_{H_{Q_8}^*} \simeq \bA_{H_{C_2}^*}$ while $\bA_{H_{D_8}^*} \not \simeq \bA_{H_{C_2}^*}$. 
\end{ex}
\begin{ex}
  Let $S^3 \langle 3 \rangle$ denote the 3-connected cover of $S^3$, then $\bA_{H^*(S^3\langle 3 \rangle)}$ has a single non-trivial object $(\Z/p,f)$ which is also central, see \cite[Example 3.7]{heard_depth} for the map $f$. Thus, this has $p$-central defect 0.
\end{ex}
For algebras of $p$-central defect 0, we have an immediate estimate for $d_0(R)$.
\begin{prop}\label{cor:central_defect_0}
  Suppose that $R \ne 0$ has $p$-central defect 0 and that \Cref{hyp:duflot} holds, then $\depth(R) = c(R)$ and
  \[
d_0(R) \le e(R) + \reg(R).
  \]
\end{prop}
\begin{proof}
  By \Cref{rem:central_defect} we have $\depth(R) = c(R)$, and then the estimate for $d_0(R)$ is an immediate consequence of the fact that $R \cong \CEss(R)$, and \Cref{thm:d0cess}.
\end{proof}

We note the following behavior of the $p$-central defect under tensor products. We once again thank the referee for the proof that the Krull dimension is additive. 
\begin{lem}
  Let $R_1$ and $R_2$ be connected Noetherian unstable algebras, then the $p$-central defect of $R_1 \otimes R_2$ is equal to the sum of the $p$-central defects of $R_1$ and $R_2$. 
\end{lem}
\begin{proof}
Using \Cref{lem:center_of_tensor} we reduce to the claim that $p(R_1 \otimes R_2) = p(R_1) + p(R_2)$. This follows directly from work of Powell \cite[Theorem 3 and Theorem 4]{powell_tensor}, along with \Cref{prop:rector_props}(2). As noted by the referee, we can also prove this directly. Firstly, if $(E_1,f_1) \in \bA_{R_1}$ and $(E_2,f_2) \in \bA_{R_2}$, then $(E_1 \otimes E_2,f_1 \otimes f_2) \in \bA_{R_1 \otimes R_2}$. It follows that $p(R_1 \otimes R_2) \ge p(R_1) + p(R_2)$.

For the converse, suppose we are given $f \colon R_1 \to H_V^* \in \bV_{R_1}$ and $g \colon R_2 \to H_V^* \in \bV_{R_2}$, so that $\widetilde f \colon R_1 \to H^*_{V/\ker(f)} \in \bA_R$ and $\widetilde g \colon R_2 \to H^*_{V/\ker(g)} \in \bA_R$. There is a commutative diagram of the form
\[
\begin{tikzcd}
R_1 \oplus R_2 \arrow[r, "\widetilde f \oplus \widetilde g"] \arrow[rd] & H^*_{V/\ker(f)} \oplus H^*_{V/\ker(g)} \arrow[r, hook] \arrow[d, two heads] & H^*_V \oplus H^*_V \arrow[d, "\mu"] \\
                                                                        & H^*_{V/(\ker(f) \cap \ker(g))} \arrow[r, hook]                              & H_V^*                              
\end{tikzcd}
\]
The composite $f \coprod g \colon R_1 \oplus R_2 \to H_V^*$ gives an element of $\bV_{R_1 \oplus R_2}$, and the commutative diagram shows that $\ker(f \coprod g) = \ker(f) \cap \ker(g)$. In particular, the rank of $V/(\ker(f) \cap \ker(g))$ is at most the sum of the ranks of $V/\ker(f)$ and $V/\ker(g)$, so that $p(R_1 \otimes R_2) \le p(R_1) \oplus p(R_2)$. Together, we see that $p(R_1 \otimes R_2) = p(R_1) + p(R_2)$, as claimed. 
\end{proof}
We can therefore construct unstable algebras of arbitrarily high $p$-central defect. For example, $H^*_{D_8^{\times n}}$ has $2$-central defect of exactly $n$. 

The unstable algebra $H^*(\stc)$ considered above is the cohomology of a $H$-space. More generally, if $X$ is a connected  Noetherian $H$-space, we recall from \Cref{ex:hspaces} that the mod 2 cohomology satisfies
\begin{equation}\label{eq:hspaceat2}
H^*(X) \cong \F_2[x_1,\ldots,x_r] \otimes \frac{\F_2[y_1,\ldots,y_s]}{(y_1^{2^{a_1}},\ldots,y_s^{2^{a_s}})}.
\end{equation}
while
\begin{equation}\label{eq:hspace_odd_2}
H^*(X;\F_p) \cong \frac{\F_p[y_1,\ldots,y_s]}{(y_1^2,\ldots,y_s^2)} \otimes \F_p[\beta y_1,\ldots, \beta y_k,x_{k+1},\ldots,x_n] \otimes \frac{\F_p[z_1,\ldots,z_t]}{(z_1^{p^{a_1}},\ldots,z_t^{p^{a_t}})},
\end{equation}
when $p$ is odd.

Both these rings are Gorenstein because they are Cohen--Macaulay, and the quotient by $(x_1,\ldots,x_r)$ for $p = 2$ (or $(\beta y_1,\ldots,x_n)$ for $p$ odd) is a Poincar\'e duality algebra of formal dimension $\sum_{i=1}^s (|y_i|^{2^{a_s-1}})$ for $p = 2$ (or $\sum_{i=1}^s |y_i| + \sum_{i=1}^t |z_i|^{p^{a_t-1}})$ for $p$ odd), see Proposition I.1.4 and the remark on the same page of \cite{MeyerSmith2005Poincare}.
\begin{defn}
  Let $X$ be a connected $H$-space with Noetherian mod 2 cohomology given as in \eqref{eq:hspaceat2}, then the Poincar\'e dimension of $H^*(X)$ is $\sum_{i=1}^s (|y_i|^{2^{a_s-1}})$.
\end{defn}
\begin{prop}\label{prop:broto_crespo}
  Let $X$ be a connected Noetherian $H$-space with cohomology as in \eqref{eq:hspaceat2}.

  If $p = 2$, then the following hold:
  \begin{enumerate}
    \item (Broto--Crespo) There exists a central elementary abelian $2$-subgroup $E$ of rank $r$ and a central morphism $f \colon H^*(X) \to H_{E}^*$. Moreover, there is a basis $u_1,\ldots,u_r$ of $H^1_E$ such that $f(x_i) = u_i^{2^{\beta_i}}$ for $\beta_i \ge 0$ and $i = 1,\ldots,r$.
    \item $H^*(X)$ has $2$-central defect 0.
      \end{enumerate}
If $p$ is odd, then the following hold:
  \begin{enumerate}
    \item (Crespo) There exists a central elementary abelian $p$-subgroup $E$ of rank $n$ and a central morphism $f \colon H^*(X) \to H_{E}^*$. Moreover, there is a basis $u_1,\ldots,u_n$ of $H^1_E$ such that 
    \[
f(y_j) = \begin{cases}
  u_j & 1 \le j \le r \quad r \le k \\
  0 & i > k
\end{cases}
    \]
    and
    \[
f(x_i) = (\beta u_i)^{p^{a_i}}. 
    \]
    \item $H^*(X)$ has $p$-central defect 0.
      \end{enumerate}
\end{prop}
\begin{proof}
Let $p = 2$, then the existence of the map $f$ and the description of the image is \cite[Theorem 2.2]{BrotoCrespo1999spaces} (where the map $f$ is denoted $\mu_X$). That $f$ is central has already been discussed in \Cref{ex:hspaces}.  Finally, $H^*(X)$ has $p$-central defect 0 because $H^*(X)$ has dimension $r$.

For $p$ odd, the same argument works using \cite[Theorem 2.6 and Corollary 2.7]{Crespo} (where the map $f$ is denoted $\ell$).  
\end{proof}
\begin{thm}\label{thm:hspace}
Suppose $X$ is a connected $H$-space with Noetherian mod p cohomology. If $X$ satisfies \Cref{hyp:duflot} and has Poincar\'e dimension $d$, then $d_0(H^*(X)) \le d$.
\end{thm}
\begin{proof}
We first consider the case $p = 2$.  Because $H^*(X)$ is Gorenstein, its local cohomology is concentrated in a single degree, namely in degree $c(H^*(X))$. From the definitions, this implies that $\reg(H^*(X)) = d + \sum_{i=1}^r(1 - |x_i|)$. Because $H^*(X)$ has $2$-central defect 0, we can then use \Cref{cor:central_defect_0} to see that
  \[
d_0(H^*(X)) \le e(H^*(X)) + d + \sum_{i=1}^r(1 - |x_i|).
  \]
  From the description of the image in the previous proposition, we have $\im(f) \cong \F_p[u_1^{2^{\beta_1}},\ldots,u_r^{2^{\beta_{r}}}]$, and so
  \[
e(H^*(X)) = \sum_{i=i}^r(2^{\beta_i}-1)  = \sum_{i=1}^r (|x_i| - 1).
  \]
  Thus,
  \[
d_0(H^*(X)) \le \sum_{i=1}^r (|x_i| - 1) + d + \sum_{i=1}^r(1 - |x_i|) = d,
  \]
  as claimed.

  The argument for $p$ odd is very similar. First we note that because \Cref{hyp:duflot} holds, we must have that $k = 0$ in \eqref{eq:hspace_odd_2}. We have $\reg(H^*(X)) = d + \sum_{i=1}^n(1-|x_i|)$, and by the assumption we can apply \Cref{cor:central_defect_0} to see that 
    \[
d_0(H^*(X)) \le e(H^*(X)) + d + \sum_{i=1}^n(1 - |x_i|).
  \]
 The Duflot algebra is then of the form $\im(f) \cong \F_p[(\beta u_{1})^{p^{a_{1}}},\ldots, (\beta u_{n})^{p^{a_{n}}}]$, and 
 \[
e(H^*(X)) = \sum_{i=1}^n (2p^{a_i}-1) = \sum_{i=1}^n (|x_i|-1).
 \]
 Thus, as above, 
   \[
d_0(H^*(X)) \le \sum_{i=1}^n (|x_i| - 1) + d + \sum_{i=1}^n(1 - |x_n|) = d,
  \]
  as claimed.
  \end{proof}
\begin{rem}
  If $p$ is odd, then we note that \Cref{prop:broto_crespo} shows that $H^*(X)$ does not always satisfy the assumptions of \Cref{hyp:duflot}. 
\end{rem}

 \subsection{The topological nilpotence degree of an unstable algebra}\label{sec:top_degree_proof}

 In this section we prove our main result (\Cref{thm:main_unstable_algebra}), which gives an estimate for $d_0(R)$. We begin with the following.
\begin{prop}\label{prop:kuhn2.7}
  For any connected Noetherian unstable algebra $R$ with center $(C,g)$ we have
  \[
d_0(R) \le \max_{(C,g) \subseteq (E,f) \in \bA_R}\{d_0(\CEss(T_E(R;f))) \}.
  \]
 \end{prop}
\begin{proof}
  Suppose that $R$ has $p$-central defect $d$. The proof will be by induction on $d$. If $d = 0$, then the statement of the proposition is clear (in fact, in this case the inequality is even an equality). Inductively, we assume that the proposition holds for all connected Noetherian unstable algebras of $p$-central defect $0 \le k < d$.

  Choose a pair $(E,f)$ with $(C,g) \subsetneq (E,f)$, and let $(C_E,\tilde g_E)$ denote the center of $T_E(R;f)$. By \Cref{prop:central_factor} there exists $h \colon T_E(R;f) \to H_E^*$ such that $(E,h)$ is central in $\bA_{T_E(R;f)}$ and the following diagram commutes:
\[
\begin{tikzcd}
R \arrow[rr, "{\rho_{R,(E,f)}}"] \arrow[rd, "f"'] &       & T_E(R;f) \arrow[ld, "h"] \\
                                                  & H_E^* &
\end{tikzcd}
\]

   By centrality, we have $(E,h) \subseteq (C_E,\tilde g_E)$, and hence (by composing with $\rho_{R,(E,f)}$) we have $(C,g) \subsetneq (E,f) \subseteq (C_E,g_E)$, where $g_E = \rho_{R,(E,f)} \circ \tilde g_E$. By \Cref{lem:central_defect} the $p$-central defect of $T_E(R;f)$ is less than that of $R$, and in particular, the inductive hypothesis applies to show that
\begin{equation}\label{eq:sub1}
d_0(T_E(R;f)) \le \max \{ d_0(\CEss(T_V(T_E(R;f);\tilde j))) \mid (C_E,\tilde g_E) \subseteq (V,\tilde j) \in \bA_{T_E(R;f)} \}.
\end{equation}
Let $j = \rho_{R,(E,f)} \circ \tilde j$, then the assumptions of \Cref{lem:center_center} are satisfied (since $(E,h) \subseteq (C_E,\widetilde g_E) \subseteq (V,\widetilde j)$), and show that 
\[
T_V(T_E(R;f);\tilde j) \cong T_V(R;j),
\]
 where $(R,j) \in \bA_R$. By \eqref{eq:sub1} we then have
\[
d_0(T_E(R;f)) \le \max \{ d_0(\CEss(T_V(R;j))) \mid (C_E,g_E) \subseteq (V,j) \in \bA_{R} \}.
\]
From the definition of the central essential ideal and \Cref{prop:dproperties}, we have
 \[
d_0(R) \le \max\{d_0(\CEss(R)),d_0(T_E(R;f)) \mid (C,g) \subsetneq (E,f) \}.
 \]
  Combining the previous two equations and observing that $(C,g) \subsetneq (C_E,g_E)$ gives the desired result.
\end{proof}

We now prove \Cref{thm:main_unstable_algebra}. For this we need to assume that $T_E(R;f)$ satisfies the assumptions of \Cref{hyp:duflot} for all $(C,g) \subseteq (E,f)$. We note that this is automatic if $p = 2$, or if $R$ is concentrated in even degrees, as then so is $T_E(R;f)$ by \Cref{lem:even_degrees_preserve}.

\begin{proof}[Proof of \Cref{thm:main_unstable_algebra}]
  Combine \Cref{thm:d0cess,thm:cessnonzero,prop:kuhn2.7}.
\end{proof}

\section{Computations of the topological nilpotence degree}
We finish with examples from group theory, and homotopical group theory, giving results analogous to Kuhn's in the case of compact Lie groups.
\subsection{Group theory}\label{sec:group_examples}
We now focus on unstable algebras of the form $R = H_G^*$ where $G$ is a group. In this case, Rector's category will take a particularly nice form. We will need the following definition.
\begin{defn}
  The Quillen category associated to a group $G$ at the prime $p$ is the category $\bA_G$ with objects elementary abelian $p$-subgroups $E \le G$ and with morphisms $E \to V$ those monomorphisms induced by conjugation in $G$.
\end{defn}
While most of the groups we study should be familiar to the reader, we first explain the class of groups considered by Broto and Kitchloo \cite{BrotoKitchloo2002Classifying}.
\begin{defn}[Broto--Kitchloo]
  Let $\cal{X}$ be a class of compactly generated Hausdorff topological groups, and let $\cal{K}_1(\cX)$ be the new class of groups, such that a compactly generated Hausdorff topological group $G$ belongs to $\cal{K}_1(\cX)$ if and only if there exists a finite $G$-CW complex $X$ with the following two properties:
  \begin{enumerate}
    \item The isotropy subgroups of $X$ belong to the class $\cX$.
    \item For every finite $p$-subgroup $\pi < G$, the fixed point space $X^{\pi}$ is $p$-acyclic.
  \end{enumerate}
\end{defn}
If $\cX$ is the class of compact Lie groups, then Kac--Moody groups are an example of a group in $\cal{K}_1(\cX)$, see \cite[Section 5]{BrotoKitchloo2002Classifying}.

With this we get the following, which is a compendium of results of Quillen \cite{Quillen1971spectrum}, Rector \cite{Rector1984Noetherian}, Lannes \cite{lannes_unpublished,lannes_ihes} Henn \cite{Henn1998Centralizers} and Broto--Kitchloo \cite{BrotoKitchloo2002Classifying}, see \cite[Theorem 4.1 and Theorem 4.8]{heard_depth} for the precise details.
\begin{thm}\label{thm:group_props}
Assume we are in one of the following cases:
\begin{enumerate}
	\item $G$ is a compact Lie group.
	\item $G$ is a discrete group for which there exists a mod $p$ acyclic $G$-CW complex with finitely many $G$-cells and finite isotropy groups.
	 	\item $G$ is a profinite group such that the continuous mod $p$ cohomology $H_G^*$ is Noetherian.
    \item $G$ is a group of finite virtual cohomological dimension such that $H_G^*$ is finite generated as an $\F_p$-algebra.
		\item $G$ is in $\cal{K}_1\cX$ where $\cX$ is the class of compact Lie groups (for example, a Kac--Moody group).
\end{enumerate}
	Then the following hold:
	\begin{enumerate}
		\item The mod $p$ cohomology $H_G^*$ is a Noetherian unstable algebra, and there is an equivalence of categories $\bA_G \simeq \bA_{H_G^*}$ given by associating to $E \le G$ the pair $(E,\res_{G,E}^*)$ where $\res_{G,E}^*$ is the restriction homomorphism $H_G^* \to H_E^*$.
		\item
		There are isomorphisms
		\[
	T_E(H_G^*;\res_{G,E}^*) \cong H_{C_G(E)}^*.
		\]
	\end{enumerate}
\end{thm}
\begin{defn}[Mislin \cite{Mislin1992Cohomologically}]
	An elementary abelian subgroup $E < G$ is said to be cohomologically $p$-central if $C_G(E) \to G$ is a mod $p$ cohomology equivalence, i.e., $H^*_G \to H_{C_G(E)}^*$ is an isomorphism.
\end{defn}
Under the equivalence of categories $\bA_G \simeq \bA_{H_G^*}$, these are precisely the central elements as considered throughout this paper (compare \Cref{ex:groups} for the case of compact Lie groups). We use the terminology cohomological $p$-central so as to not conflict with the usual group theoretic notion of central elementary abelian $p$-subgroup. The two are related in the following way, where we let $C_p(G)$ denote the maximal cohomologically $p$-central subgroup of $G$ (which is only unique up to conjugacy, see \Cref{thm:max_central}), and $Z(G)[p]$ the maximal central elementary abelian $p$-subgroup in the usual sense.
\begin{lem}[Mislin]
	If $E < G$ is a central elementary abelian $p$-subgroup, then $C_G(E)$ is cohomologically $p$-central. Moreover, there is an injective homomorphism $\phi \colon Z(G)[p] \hookrightarrow C_p(G)$.
\end{lem}
\begin{proof}
	The first claim is clear because in this case $C_G(E) \cong G$. The injective homomorphism $\phi$ is constructed exactly as by Mislin \cite{Mislin1992Cohomologically}. We recall his argument now. Let $x \in Z(G)[p]$ be represented by a map $\tilde \phi (x) \colon \Z/p \to G$, and write $f$ for the induced map $f \colon H^*_G \to H_{\Z/p}^*$. The pair $(\Z/p,f)$ is central, because $H_{G}^* \to T_{\Z/p}(H_{G}^*;f)$ corresponds to the map induced by $C_G(\langle x \rangle) \to G$. We then set $\phi(x) = f$. This is clearly injective, because if $\phi(x) = \phi(y)$, then $x$ and $y$ are conjugate in $G$, and hence equal, as they are central.
\end{proof}
\begin{rem}
  If $G$ is a finite $p$-group, then the main result of \cite{Mislin1992Cohomologically} implies that $Z(G)[p] \cong C_p(G)$, however in general $\phi$ is not surjective. A counterexample is given by the group $\Sigma_3$ at $p = 2$, as in \Cref{ex:centers3}. This means that the definition of $\CEss(H_{G}^*)$ does not necessarily agree with Kuhn's definition of $\CEss(G)$. For example, we have $\CEss(H_{\Sigma_3}^*) \cong H^*_{\Sigma_3}$ (i.e., $\CEss(\Sigma_3)$ has $p$-central defect 0), while $\CEss(\Sigma_3)$ is trivial, as it is the kernel of the restriction map $H^*_{\Sigma_3} \to H^*_{C_2}$. Of course, in any case one gets the same result, namely that $d_0(H_{\Sigma_3}^*) = 0$.
\end{rem}
\begin{thm}\label{thm:main_groups}
	Let $G$ be one of the groups considered in \Cref{thm:group_props}, then for any prime $p$ we have
	\[
d_0(H_G^*) \le  \underset{\substack{C_p(G) \le E \in \bA_G \\ \depth(H_{C_{{G}}(E)}^*) = c(C_{{G}}(E))}} \max \{e(H_{C_{{G}}(E)}^*) + \reg(H_{C_{{G}}(E)}^*)\}
	\]
  where $c(C_G(E))$ is the rank of the maximal central cohomologically $p$-central subgroup of $G$.

	Moreover, if $G$ is a compact Lie group, then $\reg(H^*_{C_G(E)}) \le -\dim(C_G(E))$, with equality if $\pi_0(C_G(E))$ is a finite $p$-group.
\end{thm}

\begin{proof}
  This will be a consequence of \Cref{thm:main_unstable_algebra,thm:group_props}, but we first explain why we are able to prove this without assuming anything about the Duflot algebra, using an observation of Nick Kuhn.\footnote{We thank Nick Kuhn for a helpful email explaining this.} The point is that for a group we can always assume that the Duflot algebra is polynomial (this has already been observed by Kuhn in the case of compact Lie groups, see \cite[Page 160]{Kuhn2013Nilpotence}).  Indeed, since the action of $G$ on $\F_p$ is trivial $H^1_G \cong \Hom_{\Z}(G,\Z/p)$ (these homomorphisms need be continuous in the case $G$ is a profinite group). In particular, elements in the image of $\res_{G,C_p(G)}^* \colon H_G^1 \to H_{C_p(G)}^1$ are exactly homomorphisms from $C_p(G) \to \Z/p$ that factor through $G$. Recall that the image of $H_G^* \to H_{C_p(G)}$ is the form
  \[
\F_p[y_1^{p^{j_1}},\ldots,y_b^{p^{j_b}},y_{b+1},\ldots,y_c] \otimes \Lambda(x_{b+1},\ldots,x_c)
  \]
Using the observation above it is not hard to see that $c-b$ is the rank of the largest subgroup of $C$ splitting off $G$ as a direct summand (compare the discussion on page 158 of \cite{Kuhn2013Nilpotence}).  Write $G = L \times E$ where $E = (\Z/p)^d$, then one sees that $d_0(H_{L \times E}^*) \cong d_0(H_L^*)$ and similar for $e_{\prim}$ and $e_{\indec}$. Thus, we can assume reduce to the case of the group $L$, which necessarily has polynomial Duflot algebra. Thus, in this case \Cref{thm:main_unstable_algebra} is valid for all primes $p$.

  Finally, the regularity statement is due to Symonds. \cite{Symonds2010CastelnuovoMumford}.
\end{proof}
\begin{rem}
	It is not true that there is always an equality $\reg(H^*_{G}) = -\dim(G)$ for a compact Lie group. For example, if $G = O(2)$ and $p$ is odd, then $\reg(H^*_{O(2)}) = -3$ (apply \cite[Lemma 1.4(2)]{Symonds2010CastelnuovoMumford}, noting that $H^*_{O(2)} \cong \F_p[x_4]$), while $-\dim(O(2)) = -1$.
\end{rem}
\begin{rem}
  Of course, one can restate this theorem as
  \[
d_0(H_G^*) \le  \max_{E < G \in \bA_G} \{e(H_{C_{{G}}(E)}^*) + \reg(H_{C_{{G}}(E)}^*)\}
  \]
  to obtain a result that is analogous to that obtained by Kuhn in the case of compact Lie groups.
\end{rem}
\begin{ex}\label{ex:gl2}
Consider the profinite group $\GL_2(\Z_3)$. This admits a splitting $\GL_2(\Z_3) \cong \Z_3 \times \GL^1_2(\Z_3)$ where $\GL^1_2(\Z/3)$ is the subgroup of $\GL_2(\Z/3)$ which is the preimage of $\Z/2 \subset \Z_3^{\times}$ under the determinant map. Moreover, $H^*_{\GL^1_2(\Z_3)} \cong (H^*_{\SL_2(\Z_3)})^{\Z/2}$. Both of these facts can be found in the proof of Proposition 5.5 of \cite{Henn1998Centralizers}. By \Cref{prop:dproperties} we obtain
\[
d_0(H_{\GL_2(\Z_3)}^*) = d_0(H_{\Z_3}^*) + d_0(H_{\GL^1_2(\Z_3)}^*) \quad \text{ and } \quad d_0(H_{\GL^1_2(\Z_3)}^*) \le d_0(H_{\SL_2(\Z_3)}^*)
\]
Because $H_{\Z_3}^* \cong \Lambda_{\F_3}(e)$ with $|e| = 1$ we can apply \Cref{prop:dproperties} again to obtain $d_0(H_{\Z_3}^*) = 1$.

In order to compute $d_0(H^*_{\SL_2(\Z_3)}$), we observe that this has a single elementary abelian subgroup $\Z/3$ whose centralizer in $\SL_2(\Z_3)$ is isomorphic to $\Z/2 \times \Z/3 \times \Z_3$, so that $H^*_{C_{\Z/3}(\SL_2(\Z_3))} \cong \F_3[y] \otimes \Lambda_{\F_3}(x,e)$, with $|y| = 2$ and $|x| = |e| = 1$, see the discussion after Proposition 5.5 (as well as Theorem 5.2) of \cite{Henn1998Centralizers}. It follows that $H^*_{\SL_2(\Z_3)}$ has trivial 3-cohomological center, and hence that $\CEss(H^*_{\SL_2(\Z_3)})$ is the kernel of the restriction map
\[
\xymatrix{H^*_{\SL_2(\Z_3)} \ar[r]& H^*_{C_{\Z/3}(\SL_2(\Z_3))} \cong \F_3[y] \otimes \Lambda_{\F_3}(x,e).}
\]
By \cite[Proposition 5.6]{Henn1998Centralizers} we deduce that $\CEss(H^*_{\SL_2(\Z_3)})$ is trivial. Since we work at $p = 3$, we have $H^*_{C_{\Z/3}(\SL_2(\Z_3))} \cong H^*_{\Z/3 \times \Z_3}$. This has depth 1, and $c(\Z/3 \times \Z_3) = 1$. Moreover, it is of $3$-central defect 0, so that $\CEss(H^*_{\Z/3 \times \Z_3}) \cong H^*_{\Z/3 \times \Z_3}$. We deduce that
\[
d_0(H^*_{\SL_2(\Z_3)} ) = d_0(H^*_{\Z/3 \times \Z_3})) = d_0(H^*_{\Z/3}) + d_0(H^*_{\Z_3}) = 0 + 1 = 1.
\]
Putting these observations together, we conclude that
\[
1 \le d_0(H^*_{\GL_2(\Z_3)}) \le 2.
\]
\end{ex}
\begin{ex}\label{ex:s2}
Consider the 2nd Morava stabilizer group $S_2$ at the prime 3. This admits a decomposition $S_2 \cong S_2^1 \times \Z_3$ and so
\[
d_0(H_{S_2}^*) = d_0(H_{S_2^1}^*) + d_0(H_{\Z_3}^*) = d_0(H_{S_2^1}^*) + 1.
\]
The group $S_2^1$ has two conjugacy classes of elementary abelian 3-subgroups $E_i$ for $i = 1,2$ with $C_{S_2^1}(E_i) \cong \Z/3 \times \Z_3$ in both cases. We note that both $S_2^1$ and $C_{S_2^1}(E_i)$ are 3-adic Lie groups. We also observe that $H_{S_2^1}^*$ has trivial 3-cohomological center, and hence $\CEss(H_{S_2^1}^*)$ is the kernel of the product of restriction maps
\[
\xymatrix{H_{S_2^1}^* \ar[r] & \prod_i H^*_{C_{S_2^1}(E_i)} \cong \prod_{i = 1}^2 \F_3[y_i] \otimes \Lambda_{\F_3}[x_i,e_i].}
\]
By \cite[Proposition 4.3]{Henn1998Centralizers} we deduce that $\CEss(H_{S_2^1}^*)$ is trivial. One then deduces from \Cref{prop:dproperties} that
$d_0(H^*_{S_2^1}) \le d_0(H^*_{C_{S_2^1}(E_1)})$ (of course, one can use either $E_1$ or $E_2$ here). However, the $T$-functor computations show that $H^*_{C_{S_2^1}(E_1)}$ is a summand of $T_{E_1}(H^*_{S_2^1})$ and so $d_0(H^*_{S_2^1}) \ge d_0(H^*_{C_{S_2^1}(E_1)})$. Thus, $d_0(H^*_{S_2^1}) = d_0(H^*_{C_{S_2^1}(E_1)}) = d_0(H_{\Z/3 \times \Z_3}^*) = 1$. We deduce that $d_0(H_{S_2}^*) = 2$.
\end{ex}
\subsection{Homotopical groups}\label{sec:homotopical}
We now move onto the case of homotopical groups, namely the $p$-local finite and compact groups of Broto, Levi, and Oliver \cite{BrotoLeviOliver2003homotopy,BrotoLeviOliver2007Discrete}. Once we have set up the right language, the results take essentially the same form as for ordinary groups. The canonical references for both $p$-local finite and compact groups are the aforementioned papers of Broto, Levi, and Oliver, however the reader may also find the survey article \cite{BrotoLeviOliver2004theory} valuable.

To begin, we recall the definition of the fusion system $\cal{F}_p(G)$ associated to a finite group $G$. This is a category whose objects are the $p$-subgroups of $G$, and where
\[
\Hom_{\cal{F}_p(G)} =   \Hom_G(P,Q) \coloneqq \{ \alpha \in \Hom(P,Q) \mid \alpha = c_x, \text{ for some } x \in G \}.
\]
i.e., $\alpha$ is a homomorphism induced by conjugation in $G$. To this one can associate another category, the centric linking system $\cal{L}_p^c(G)$. Then, by \cite[Proposition 1.1]{BrotoLeviOliver2003Homotopy2} there is a homotopy equivalence $|\cal{L}_p^c(G)|^{\wedge}_p \simeq BG^{\wedge}_p$.

The idea of $p$-local finite groups is to begin with a finite $p$-group $S$, and try and mimic the constructions above. Thus, a fusion system $\cal{F}$ associated to $S$ is a category whose objects are subgroups of $S$, and whose morphism sets $\Hom_{\cal{F}}(P,Q)$ satisfy the following conditions:
\begin{enumerate}
  \item $\Hom_S(P,Q) \subseteq \Hom_{\cal{F}}(P,Q) \subseteq \Inj(P,Q) \text{ for all } P,Q \le S$.
  \item Every morphism in $\cal{F}$ factors as an isomorphism in $\cal{F}$ followed by an inclusion.
\end{enumerate}
This is not quite enough; Broto, Levi, and Oliver additionally require that the fusion system is \emph{saturated}, see \cite[Definition 1.2]{BrotoLeviOliver2003homotopy}. A centric linking system $\cal{L}$ associated to $\cal{F}$ is another category whose objects are a certain subset of $S$. The centric linking system contains the additional data to associate a classifying space to the fusion system $\cal{F}$.

A \emph{$p$-local finite group} is a triple $\cal{G} = (S,\cal{F},\cal{L})$ where $S$ is a finite $p$-group, $\cal{F}$ is a saturated fusion system over $S$, and $\cal{L}$ is a centric linking system associated to $\cal{F}$. The classifying space of $\cal{G}$ is defined as $B\cal{G} = |\cal{L}|^{\wedge}_p$, the $p$-completed nerve of the category $\cal{L}$. We write $H_{\cal{G}}^* \coloneqq H^*(B\cal{G})$ for the mod $p$ cohomology of $\cal{G}$.

If instead of a finite $p$-group we begin with a discrete $p$-toral group $S$ - that is a group that contains a normal subgroup $T \cong (\Z/p^{\infty})^{r}$, and such that $T$ has finite index in $S$ - then we can define saturated fusion systems $\cal{F}$ over $S$, and centric linking systems over $\cal{F}$, see \cite{BrotoLeviOliver2007Discrete}. A \emph{$p$-local compact group} is a triple $\cal{G} = (S,\cal{F},\cal{L})$ where $S$ is a discrete $p$-toral group, $\cal{F}$ is a saturated fusion system over $S$, and $\cal{L}$ is a centric linking system associated to $\cal{F}$. In fact, it was later shown that every saturated fusion system over a discrete $p$-toral group has an associated centric linking system which is unique up to isomorphism \cite{Chermak2013Fusion,LeviLibman2015Existence}. Thus, we often define our $p$-local compact groups as simply a pair $\cal{G} = (S,\cal{F})$.

\begin{ex}
  Here we list some examples of $p$-local compact groups.
  \begin{enumerate}
    \item If $G$ is a compact Lie group, with $p$-toral subgroup $S$, then there exists a $p$-local compact group $\cal{G} = (S,\cal{F}_S(G))$ along with an equivalence of classifying spaces $BG^{\wedge}_p \simeq B\cal{G}$ \cite[Theorem 9.10]{BrotoLeviOliver2007Discrete}. \qed
    \item Suppose that $X$ is a $p$ compact group, that is a triple $(X,BX,e)$ where $X$ is a space with $H^*(X;\F_p)$ finite, $BX$ is a pointed $p$-complete space, and $e \colon X \to \Omega(BX)$ is an equivalence \cite{DwyerWilkerson1994Homotopy}. There is a notion of a Sylow subgroup $f \colon S \to X$, and moreover, there exists a $p$-local compact group $\cal{G} = (S,\cal{F}_{S,f}(X))$ with $B\cal{G} \simeq BX$ \cite[Theorem 10.7]{BrotoLeviOliver2007Discrete}. More generally, the $p$-completion of any finite loop space gives rise to a $p$-local compact group \cite{BrotoLeviOliver2014algebraic}.
  \end{enumerate}
\end{ex}
\begin{rem}
  Because, up to $p$-completion, every compact Lie group can be modeled by a $p$-local compact group, this section recovers the result for compact Lie groups in the previous section. This follows because the classifying space of a compact Lie group is always $p$-good (see \cite[Proposition VII.5.1]{BousfieldKan1972Homotopy}) and so $H_{G}^* \cong H^*(BG^\wedge_p)$.
\end{rem}

One has an immediate upper bound for $d_0(H_{\cG}^*)$ coming from the group $S$ in the case of $p$-local finite groups. 
\begin{prop}\label{prop:transfer}
  Let $\cal{G} = (S,\cF)$ be a $p$-local finite group, then $d_0(H_{\cG}^*) \le d_0(H_S^*)$. 
\end{prop}
\begin{proof}
  By \cite[Proposition 5.5]{BrotoLeviOliver2003homotopy} $H_{\cG}^*$ is a direct summand of $H_S^*$ as an unstable algebra, and so \Cref{prop:dproperties} furnishes the result. 
\end{proof}
Even for finite groups, it is known that this inclusion is strict in general. For example, Henn--Lannes--Schwartz \cite[Section II.4.7]{HennLannesSchwartz1995Localizations} give the example $G = GL_2(\F_3)$ and $S = SD_{16}$, then $d_0(H_G^*) = 0$, but $d_0(H_S^*) = 2$. 

We now identify Rector's category $\bA_{H_{\cal{G}}^*}$ and present the relevant $T$-functor calculations. We let $\cal{F}^e$ be the full subcategory of $F$ whose objects are elementary abelian $p$-subgroups $E \le S$ which are fully-centralized in $\cal{F}$ in the sense of \cite[Definition 2.2]{BrotoLeviOliver2007Discrete}. This assumption ensures that the centralizer $p$-compact group $C_{\cal{G}}(E) = (C_S(E),C_{\cal{F}}(E))$ exists \cite[Section 1.2]{Gonzalez2016Finite}, where $C_{\cal{F}}(E)$ is the fusion system over $C_S(E)$ with objects $Q \le C_S(E)$ and morphisms
\[
\Hom_{C_{\cal{F}}(E)}(Q,Q')= \{ \psi \in \Hom_{\cal{F}}(Q,Q') \mid \exists \phi \in \Hom_{\cal{F}}(QE,Q'E),\phi |_Q = \psi, \phi|_E = \text{id}_E\}.
\] Moreover, we note that any elementary abelian $p$-subgroup $E \le S$ is isomorphic in $\cal{F}$ to one that is fully $\cal{F}$-centralized.

For the following, we note that there is a canonical map $\theta \colon BS \to B\cal{G}$. Then, for any $E \le S$, there is a map $j_E \colon H_{\cal{G}}^* \to H_E^*$ given as the composite $H_{\cal{G}}^* \xr{\theta^*} H_S^* \to H_E^*$.
\begin{prop}\label{prop:p_local_rector}
  Let $\cal{G} = (S,\cal{F})$ be a $p$-local compact group, then $H_{\cal{G}}^*$ is a finitely-generated $\F_p$-algebra, and there is an equivalence of categories
  \[
\bA_{H_{\cal{G}}^*} \simeq \cal{F}^e
  \]
  given by associating to a fully centralized subgroup $E \le S$ the pair $(E,j_E)$.
\end{prop}
\begin{proof}
  The finite generation is \cite[Corollary 4.26]{bchv}, while the identification of $\bA_{H_{\cal{G}^*}}$ is \cite[Proposition 4.18]{heard_depth}.
\end{proof}

Let $E < S$ be an elementary abelian $p$-subgroup, then the map $C_S(E) \times E \to S$ induces $H_S^* \to H_E^* \otimes H_{C_S(E)}^*$. The adjoint induces a map $\phi_E \colon T_E(H_S^*;\res_{S,E}^*) \to H_{C_S(E)}^*$. The following is shown by Gonzalez \cite[Lemma 5.1]{Gonzalez2016Finite}.
\begin{lem}\label{lem:t_functor_p_local}
  For any $E \in \cal{F}^e$ there is an isomorphism
  \[
T_E(H_{\cal{G}}^*;j_E) \xr{\cong} H_{C_{\cal{G}}(E)}^*
  \]
  which is the restriction of the homomorphism $\phi_E$.
\end{lem}
For $E \in \cal{F}^e$ a special case of \cite[Theorem 5.4]{Gonzalez2016Finite} identifies $BC_{\cal{G}}(E)$ with $\Map(BE,B\cal{G})_{B\iota}$ where $\iota \colon E \to S$, and so $H_{C_{\cal{G}}(E)}^* \cong H^*(\Map(BE,B\cal{G})_{B\iota})$. Under this, the map $\rho_{H_{\cal{G}}^*,(E,j_E)} \colon H^*_{\cal{G}} \to T_E(H_{\cal{G}}^*;j_E)$ can be identified with the map induced by the evaluation map $\Map(BE,B\cal{G})_{B\iota} \to B\cal{G}$.
\begin{defn}
	Let $\cal{G} = (S,\cal{F})$ be a $p$-local compact group, then $E \in \cal{F}^e$ is called central if $\Map(BE,B\cal{G})_{B\iota} \to B\cal{G}$ is a homotopy equivalence.
\end{defn}
This does not conflict with the notion of centrality used previously in this paper, by the following lemma.
\begin{lem}\label{lem:centrality_p_local}
	$E \in \cal{F}^e$ is central if and only if $\rho = \rho_{H_{\cal{G}}^*,(E,j_E)} \colon H_{\cal{G}}^* \to T_E(H_{\cal{G}}^*;j_E) \cong H^*_{C_{\cal{G}}(E)}$ is an equivalence. In other words, $E \in \cal{F}^e$ is central if and only if $(E,j_E) \in \bA_{H_{\cal{G}}^*}$ is central. 
\end{lem}
\begin{proof}
If $E \in \cal{F}^e$ is central then this is clear from the discussion before the definition of centrality. For the converse, suppose that $\rho$ is an equivalence. Because the classifying space of a $p$-local compact group is $p$-good (combine \cite[Proposition 4.4]{BrotoLeviOliver2007Discrete} and \cite[Proposition I.5.2]{BousfieldKan1972Homotopy}), the map $\Map(BE,B\cal{G})_{B\iota} \to B\cal{G}$ is a homotopy equivalence.
\end{proof}
\begin{rem}
   By \cite[Theorem 7.4]{BrotoLeviOliver2007Discrete} if $E \in \cal{F}^e$ is central, then the $p$-local compact groups $\cal{G}$ and $C_{\cal{G}}(E)$ are isomorphic in the sense discussed on \cite[pp.~374-375]{BrotoLeviOliver2007Discrete}. In particular, there are isomorphisms of groups and categories $\alpha \colon S \to C_S(E)$ and $\alpha_{\cal{F}} \colon \cal{F} \to C_{\cal{F}}(E)$ which are compatible in a certain sense.
\end{rem}
Note that we have a natural definition of $p$-centrality for a $p$-local compact group. For the following, we let $C(\cal{G})$ denote the maximal central elementary abelian $p$-subgroup $E \in \cal{F}^e$, which exists by \Cref{thm:max_central} and the previous lemma.
\begin{defn}
  Let $\cal{G} = (S,\cal{F})$ be a $p$-local compact group, then $\cal{G}$ is $p$-central if the $p$-rank of $\cal{F}^e$ (i.e., the rank of a maximal elementary abelian $p$-group in $\cal{F}^e$) is equal to the rank of $C(\cal{G})$.
\end{defn}
\begin{lem}\label{lem:p_local_central_defect_0}
  $\cal{G}$ is $p$-central if and only if $H_{\cal{G}}^*$ has $p$-central defect 0.
\end{lem}
\begin{proof}
  This follows immediately from the definition and \Cref{lem:centrality_p_local}.
\end{proof}

\begin{thm}\label{thm:main_plocalcompact}
	Let $\cal{G} = (S,\cal{F})$ be a $p$-local compact group and assume that $H_{C_{\cal{G}(E)}}^*$ satisifies \Cref{hyp:duflot} for all $E \in \cal{F}^e$, then
	\[
d_0(H_{\cal{G}}^*) \le  \underset{\substack{C(\cal{G}) \le E \in \cal{F}^e \\ \depth(H_{C_{\cal{G}}(E)}^*) = c(C_{\cal{G}}(E))}} \max \{e(H_{C_{\cal{G}}(E)}^*) + \reg(H_{C_{\cal{G}}(E)}^*)\}
	\]
	If $S$ is a finite $p$-group (i.e., $\cal{G}$ is a $p$-local finite group), then $\reg(H_{C_{\cal{G}}(E)}^*) = 0$.
\end{thm}
\begin{proof}
	Combine \Cref{thm:main_unstable_algebra} with \Cref{prop:p_local_rector,lem:t_functor_p_local}. The computation of the regularity is due to Symonds \cite[Proposition 6.1]{Symonds2010CastelnuovoMumford} and Kessar--Linckelmann \cite[Theorem 0.4]{KessarLinckelmann2015CastelnuovoMumford}.
\end{proof}
\begin{rem}
	Currently, there is only a single example of an exotic family of 2-local finite groups, namely the Solomon 2-local compact groups $\Sol(q)$ for $q$ an odd prime power, where $S$ is a Sylow $2$-subgroup of $\Spin_7(q)$ \cite{LeviOliver2002Construction}. Grbi\'{c} \cite[Proposition 2]{Grbic2006cohomology} has shown that $H^*_{\Sol(q)} $ has $H^*_{DI(4)}$ as a split summand in the category of unstable algebras, where $DI(4)$ is the exotic 2-compact group of Dwyer and Wilkerson \cite{DwyerWilkerson1993new}. By \Cref{prop:dproperties} we have $d_0(H^*_{\Sol(q)}) \le d_0(H^*_{DI(4)})$.	Because $H^*_{DI(4)}$ realizes the mod 4 Dickson invariants, there is an inclusion $H^*_{DI(4)} \to H_{(\Z/2)^4}^*$ of unstable algebras, so that $d_0(H^*_{DI(4)}) = 0$. 
		\[
d_0(H^*_{\Sol(q)}) = 0.
	\]
	Unfortunately, the relevant calculations for the centralizer $2$-local finite groups are not known, so we cannot compare this to the estimate from \Cref{thm:main_plocalcompact}.

  At odd primes, the $p$-local finite groups with $S = p_+^{1+2}$, the extraspecial groups of order $p^3$ and exponent $p$, have been calculated by Ruiz and Viruel \cite{RuizViruel2004classification}. In particular, at $p = 7$, they construct three new, exotic, examples of $p$-local finite groups. By \Cref{prop:transfer} we have $d_0(H_{\cal{G}}^*) \le d_0(H_{S}^*)$ for any of these three groups. By \cite[Theorem 13.21]{Totaro2014Group} we can deduce that $d_0(H_{\cal{G}}^*) \le 4$. 
\end{rem}
\appendix

\section{Borel equivariant cohomology}\label{sec:borel}
We recall from \Cref{sec:nilpotent_filtratio} that the work of Henn--Lannes--Schwartz shows that if $R$ is a Noetherian unstable algebra, and $M \in  R_{fg}-\cal{U}$, then $d_0(M)$ is finite, and that
\[
\phi_M \colon M \to \prod_{(E,f) \in \bA_R} H_E^* \otimes T_E(M;f)^{\le n}
\]
is injective for $n \ge d_0(M)$. In this work, we have focused on the case where $M = R$. In this appendix, we specialize to the case where $R = H_G^*$ for a compact Lie group $G$, and $M = H_G^*(X)$ for $X$ a manifold. As in \Cref{ex:borel} $M \in R_{fg}-\cal{U}$ by Quillen \cite{Quillen1971spectrum}. In this case, using \cite{lannes_unpublished} (see also \cite[Theorem 2.6]{henn_notes}) the previous equation takes the form
\begin{equation}\label{eq:borel}
H_G^*(X) \to \prod_{E \le G} H_E^* \otimes H^{\le n}_{C_G(E)}(X^E),
\end{equation}
see \cite[Theorem 5.5]{HennLannesSchwartz1995Localizations}.

It is worth explain how the maps in this theorem arise (following the discussion on \cite[p.~48]{HennLannesSchwartz1995Localizations}). The canonical homomorphism $E \times C_G(E) \to C_G(E)$ induces a map $BE \times (EC_G(E) \times_{C_G(E)} X^E) \to EC_G(E) \times_{C_G(E)} X^E$. We then define $c_E$ as the composite of the previous map with the map $EC_G(E) \times_{C_G(E)} X^E \to EG \times_{G} X$. The induced maps
\[
c_E^* \colon H_G^*(X) \to  H_E^* \otimes H^*_{C_{G}(E)}(X^E)
\]
induce the homomorphism in \eqref{eq:borel} as $E$ runs over the elementary abelian $p$-subgroups of $G$.

 We will show that slight adaptations of our techniques hold in this case. The observation we use here is that we have a good notion of centrality in this case. Indeed, suppose that $E,V$ are central subgroups of $G$ acting trivially on $X$. Then the subgroup $E \circ V$ generated by $E$ and $V$ is still central, and also acts trivially on $X$. Thus, there is a maximal central subgroup of $G$ that acts trivially on $X$. Throughout this section, we let $C=C(G,X)$ denote this maximal central subgroup, and let $e(G,X)$ denote the top degree of a generator of the finitely generated $H_G^*$-module $H_C^*$. 

  We observe (see \cite{BrotoHenn1993Some}) that $H_G^*(X)$ is a $H_C^*$-comodule and that the image of the restriction map $H_G^*(X) \to H_C^*$ is a sub-Hopf algebra of $H_C^*$. Applying the Borel structure theorem, we can identify the image of this exactly as in \Cref{cor:Borel}. 

The central essential ideal is defined in the obvious way, namely as the kernel
\[
\xymatrix{
0 \ar[r] & \CEss(G,X) \ar[r] &  H^*_G(X) \ar[r] & \displaystyle \prod_{C(G,X) \lneq E} H^*_{C_G(E)}(X^E),
}
\]
One deduces, as in \Cref{thm:krulldimension} that the Krull dimension of $\CEss(G,X)$ is at most the rank of $C$. The regularity of $H_G^*(X)$ is also known in this case; it is a theorem of Symonds \cite[Theorem 0.1]{Symonds2010CastelnuovoMumford} that $\reg(H_G^*(X)) \le \dim(X) - \dim(G)$.

 The same argument as in the body of the paper then gives the following result. 
\begin{thm}\label{thm:borel_appendix}
  Let $G$ be a compact Lie group, $X$ a manifold, and suppose that the Duflot algebra for $H_{C_G(E)}^*(X^E)$ is polynomial for all $C(G;X) \le E$, then
  \[
d_0(H_G^*(X)) \le \max_{C(G,X) \le E < G}\{e(C_G(E),X^E) + \dim(X^E) - \dim(C_G(E)) \}
  \]
\end{thm}

\section{Depth and dimension}\label{sec:appendix}
In this appendix we briefly recall the notions of depth and dimension of graded-commutative connected Noetherian $k$-algebras for $k$ a field. Given such a $k$-algebra we write $R^j$ for the degree $j$ part of $R$. Hence, $R$ connected means that $R^0 \cong \F_p$ and $R^i = 0$ for  $i<0$. We let $\frak m = R^{> 0}$ denote the maximal homogeneous ideal of $R$. With these assumptions, the commutative algebra of $R$ is much like that of a local ring. We will follow the convention that, unless noted otherwise, everything is taken in the graded sense and ideals and elements of $R$-modules are always taken to be homogeneous.
\begin{defn}
	The Krull dimension of $R$, denoted $\dim(R)$ is the supremum of lengths of strictly increasing chains of prime ideals. For an $R$-module $M$, the dimension of $M$, $\dim_R(M)$ is defined as the dimension of $R/\Ann_R(M)$, where $\Ann_R(M) = \bigcap_{m \in M}\Ann_R(m)$ and
	\[
\Ann_R(m) = \{r \in R \mid rm= 0\}.
	\]
\end{defn}
\begin{defn}
  Let $M$ be an $R$-module, then an $M$-regular sequence is a sequence $y_1,\ldots,y_m$ in $\frak m$ such that $y_i$ is a non-zero divisor on $M/(y_i,\ldots,y_{i-1})$ for $i = 1,\ldots,m$. If $M$ is finitely generated over $R$ then the depth of $M$, denoted $\depth_R(M)$ is the supremum of the length of all $M$ -regular sequences in $\frak m$.
\end{defn}
We have the following useful characterization of $M$-regular sequences, see \cite[Proposition 12.2.1]{CarlsonTownsleyValeriElizondoZhang2003Cohomology}.
\begin{lem}\label{lem:depth_regular}
  Let $M$ be a finitely-generated $R$-module. A sequence $y_1,\ldots,y_m \in \frak m$ of homogeneous elements of $\frak m$ is an $M$-regular sequence if and only if $y_1,\ldots,y_m$ are algebraically independent in $R$ and $M$ is a free module over the polynomial subring $k[y_1,\ldots,y_m] \subseteq R$.
\end{lem}
We recall that the $\frak m$-torsion in $M$ is
\[
H_{\frak m}^0(M) = \{ x \in M \mid \text{ there exists } n \in \mathbb{N} \text{ with } m^{\frak n}x = 0\}.
\]
This functor is left exact, and we let $H_{\frak m}^i(M)$ denote the higher derived functors, which are the local cohomology modules of $M$. Depth and dimension are related to local cohomology in the following way, see \cite[Corollary 6.2.8]{BrodmannSharp2013Local}.
\begin{prop}\label{prop:depth_local_cohom}
Suppose that $R$ is Noetherian and connected, and let $M$ be a finitely generated $R$-module.
  \begin{enumerate}
    \item The depth of $M$ is the smallest $i$ for which $H_{\frak m}^i(M) \ne 0$.
    \item The dimension of $M$ is the largest $i$ for which $H_{\frak m}^i(M) \ne 0$.
  \end{enumerate}
\end{prop}
The characterization of depth in terms of local cohomology, and the independence theorem for local cohomology (see \cite[Theorem 14.1.7]{BrodmannSharp2013Local}) give the following.
\begin{lem}\label{lem:fgdepth}
  Let $R$ and $R'$ be connected Noetherian unstable algebras, and $f \colon R \to R'$ a finite homomorphism. Let $M$ be a finitely-generated $R'$-module, then
  \[
  \depth_R(M) = \depth_{R'}(M)
  \]
  where $M$ is an $R$-module by restriction of scalars. In particular,
  \[
\depth_R(R') = \depth(R').
  \]
\end{lem}
Finally, we will need the following version of a group theoretic theorem of Carlson \cite{CarlsonTownsleyValeriElizondoZhang2003Cohomology}, which is proved by the author in \cite[Theorem 3.5]{heard_depth}.
\begin{thm}\sloppy\label{thm:duflot_regular}
  Let $R$ be a connected Noetherian unstable algebra, and suppose $(E,f) \in \bA_R$ is central. If $x_1,\ldots,x_n$ is a sequence of homogeneous elements in $R$ such that the restrictions of $x_1,\ldots,x_n$ form a regular sequence in $H_E^*$, then $x_1,\ldots,x_n$ is a regular sequence in $R$.
  \end{thm}
An easy consequence is the following, see \cite[Corollary 3.6]{heard_depth}, which was originally proved in the case $R = H_G^*$ for $G$ a finite group by Duflot \cite{Duflot1981Depth}.
\begin{cor}\label{cor:duflot}
	Let $R$ be a Noetherian unstable algebra with center $(C,g)$, then
	\[
\depth(R) \ge \rank(C).
	\]
\end{cor}
We will say that $R$ has minimal depth if $\depth(R) = \rank(C)$.
\bibliographystyle{alpha}
\bibliography{nilpotence}

\end{document}